\documentclass[leqno,twoside,11pt]{amsart}

\usepackage{latexsym,esint,url}
\usepackage{color}

\theoremstyle{plain}
 \def\endproof{\hspace*{\fill}\mbox{\ \rule{.1in}{.1in}}\medskip }

\newtheorem{theorem}{Theorem}[section]
\newtheorem{corollary}[theorem]{Corollary}
\newtheorem{lemma}[theorem]{Lemma}

\theoremstyle{definition}
\newtheorem{example}[theorem]{Example}
\newtheorem{remark}[theorem]{Remark}

\numberwithin{equation}{section}
\numberwithin{figure}{section}

\begin{document}

\title[Plates with incompatible prestrain]
{Plates with incompatible prestrain}
\author{Kaushik Bhattacharya, Marta Lewicka and Mathias Sch\"affner}
\address{Kaushik Bhattacharya, Division of Engineering and Applied
  Science, California Institute of Technology, Pasadena, CA 91125}
\address{Marta Lewicka, University of Pittsburgh, Department of Mathematics, 
139 University Place, Pittsburgh, PA 15260}
\address{Mathias Sch\"affner, University of W\"urzburg, Institute for
  Mathematics, Emil-Fischer-Str. 40, 97074 W\"urzburg, Germany}
\email{bhatta@caltech.edu, lewicka@pitt.edu, mathias.schaeffner@mathematik.uni-wuerzburg.de}

\date{\today}
\begin{abstract} 
We study the effective elastic behavior of incompatibly prestrained plates, 
where the prestrain is independent of thickness as well as uniform through the
thickness.  We model such plates as three-dimensional elastic  bodies 
with a prescribed pointwise stress-free state 
characterized by a Riemannian metric $G$ with the above properties,
and seek the limiting behavior as the thickness goes to zero.

We first establish that the $\Gamma$-limit is a Kirchhoff-type bending
theory when the energy per volume scales as the second power of thickness. 
We then show the somewhat surprising result that there are metrics 
which are not immersible, but have zero bending energy.  This 
implies that there are a hierarchy of plate theories for such prestrained
plates.  We characterize the non-immersible metrics that have zero bending energy
 (if and only if  the  Riemann curvatures $R^3_{112},
R^3_{221}$ and $ R_{1212}$   of $G$ do not identically vanish), 
and illustrate them with examples.  Of particular interest is an
example where $G_{2\times 2}$, the two-dimensional restriction of $G$ 
is flat, but the plate still has non-trivial energy that scales similar
to F\"oppl - von K\'arm\'an plates. Finally, we apply these results to a model of nematic glass,
including a characterization of the condition when the metric is immersible,
for $G=\mbox{Id}_3 +\gamma\vec n\otimes \vec n$ given 
in terms of the inhomogeneous unit director field 
distribution $\vec n\in\mathbb{R}^3$. 
\end{abstract}

\maketitle

\section{Introduction}

There are a number of phenomena where thin plates become prestrained
in an inhomogeneous and incompatible manner.  In such situations,
the plates become internally stressed, deform out of plane and assume
non-trivial three dimensional shapes.  Growing leaves, gels
subjected to differential swelling, electrodes in electrochemical
cells, edges of torn plastic sheets are but a few examples  
(see \cite{sharon,lemapa1} and referenced there).  It has also been
recently suggested that such incompatible prestrains may be exploited
as means of actuation of micro-mechanical devices \cite{MWB2,MWB}.

The F\"oppl - von K\'arm\'an plate theory has been widely used in the
literature to study incom\-pa\-ti\-ble pre-strained-induced bending.
It has recently been shown \cite{lemapa1} that such a theory 
arises as a Gamma-limit of the three-dimensional energy if the prestrain 
identity tends to zero together with the plate's thickness.  However,
 there are situations like those of liquid crystalline solids  where the pre-strain 
 is not small.  Therefore, a proper formulation of such problems is of interest.

A possible mathematical foundation 
relies on a model referred to as ``non-Euclidean'' theory of elasticity.
This model postulates that the three dimensional elastic body seeks to realize a configuration with 
a prescribed Riemannian metric $G$. Although there always exists a Lipschitz isometric immersion of any
$G$ \cite{gromov}, 
any such immersion is necessarily neither orientation preserving nor reversing in any neighborhood of a point
where the Riemann curvature of $G$ does not vanish (i.e. when the metric is non-Euclidean). 
Excluding such  unphysical deformations leads to 
the model potential (or elastic) energy $E$ which measures how far a
given deformation $u$ is from being an orientation preserving realization 
of $G$.   The infimum of $E$ in absence of any forces or boundary 
conditions is indeed strictly positive for any non-Euclidean $G$
\cite{lepa}, and this points to the existence of non-zero stress at
free equilibria.   Of particular interest are situations where the domain occupied by the
body is small in one dimension compared to the other two.  In such situations,
it is of interest to establish asymptotic models.

To be specific, let $\Omega\subset\mathbb{R}^2$ be an open, bounded, smooth and
simply connected set. For small $h>0$ we consider thin plates with mid-plate
$\Omega$, given by:
$$\Omega^h = \Omega\times (-\frac{h}{2}, \frac{h}{2}) = \left\{ x=
  (x', x_3);~ x'\in\Omega, ~~ |x_3|<\frac{h}{2}\right\}.$$
Let $G:\bar\Omega\rightarrow\mathbb{R}^{3\times 3}$ be  a smooth field of
symmetric positive definite matrices, so that it defines a Riemannian
metric on $\Omega^h$.
In this paper we study the situation where $G$ is
independent of the thickness as well and uniform through the
thickness.  Specifically, we assume 
\begin{equation} \label{g}
G(x', x_3)  = G(x')= [G_{ij}(x')]_{i,j=1\ldots 3} \qquad \forall (x', x_3)\in \Omega^h.
\end{equation}
We consider the following ``non-Euclidean energy'' functional:
\begin{equation}\label{energy}
E^h(u^h) = \frac{1}{h}\int_{\Omega^h} W(\nabla u^h A^{-1})~\mbox{d}x 
\qquad \forall u^h\in W^{1,2}(\Omega^h,\mathbb{R}^3),
\end{equation}
where $A$ is the positive definite symmetric square root of $G$:
$$A = \sqrt{G},$$
while $W:\mathbb{R}^{3\times 3}\longrightarrow
\bar{\mathbb{R}}_{+}$ is the elastic energy density. 
In addition to being $\mathcal{C}^2$ regular in a neighborhood of
$SO(3)$, the density $W$ is assumed to satisfy the normalization, 
frame indifference and non\-de\-ge\-ne\-racy conditions as below:
\begin{equation}\label{cond}
\begin{split}
\exists c>0\quad \forall F\in \mathbb{R}^{3\times 3} \quad
\forall R\in SO(3) \qquad
&W(R) = 0, \quad W(RF) = W(F),\\
&W(F)\geq c~ \mathrm{dist}^2(F, SO(3)).
\end{split}
\end{equation}
We are interested in understanding the limiting behavior of  $\inf E^h$ as 
$h \rightarrow 0$.

It follows from the classical work of Nash and Kuiper \cite{na1, na2},
and from Gromov's convex integration method \cite{gromov}, that any smooth metric $G_{2\times 2}$ on $\Omega$
has a $\mathcal{C}^{1,\alpha}$ (with $\alpha<1/7$) isometric immersion
in $\mathbb{R}^3$ (see also \cite{bo, del}). 
In view of the work  of Le Dret and Raoult \cite{lera}, if this immersion is smooth, we expect that 
the highest non-zero scaling of the energy $E^h$ is  $\frac{1}{h^2} E^h$ (in fact this is shown in 
Corollary 3.2 below, with the critical regularity $W^{2,2}$).  
Therefore, we focus on this  and smaller scaling of the energy.

\medskip

We explore three issues. 

First, in Part A, we establish (Theorems \ref{thm2}
and \ref{thm3}) that  the $\Gamma$-limit of $\frac{1}{h^2} E^h$
is given by a Kirchhoff bending energy functional (\ref{IG}) acting on the set of all $W^{2,2}$
realizations of $G_{2\times 2}$ in $\mathbb R^3$.   The main tool 
is the geometric rigidity theorem of Friesecke, James and M\"uller \cite{FJMhier}.
In fact, these authors used this theorem to rigorously derive the  nonlinear
Kirchhoff  bending theory as the $\Gamma$-limit of the classical Euclidean
elasticity theory ($G = \mbox{Id}_3$) under the assumption that
$E^h$ scales like $h^2$.  Lewicka and Pakzad \cite{lepa} extended their approach to 
the non-Euclidean setting
in the  particular case $G = G_{2\times 2}^* +  e_3\otimes e_3$. 
Our results extend this to the case of arbitrary $G$ that satisfy (\ref{g}).

Second, in Part B, we  show that the situation (\ref{g}) can lead to 
other non-trivial theories.  We first show (Theorems \ref{th12} and 
\ref{th23}) that the limit of the infimum of $\frac{1}{h^2}
E^h$ is non-zero if and only if the three Riemann curvatures
$R^3_{112}, R^3_{221}$ and $ R_{1212}$ of $G$ do not vanish
identically.  Therefore, there exist non-immersable metrics $G$ for 
whom the minimum of their  bending energy is zero.  
Therefore, these metrics lead to non-trivial theories at scalings 
smaller than $h^2$.   While the detailed derivation of such theories
remain open, we explore them through Examples \ref{ex1} through \ref{ex3}.
We show in Example \ref{ex1} that the infimum of $E^h$ scales as $h^4$
(Theorem \ref{optimal}).  This is reminiscent of the scaling of 
F\"oppl - von K\'arm\'an theories.  However, in this example
$G_{2\times 2}$ is flat and therefore the limiting deformation involves
no bending!    In other words, these can be regarded as a 
generalized plane stress with F\"oppl - von K\'arm\'an scaling.
Further, we show in Example \ref{ex3} that this residual energy is not 
limited to cases where $G_{2\times 2}$ is flat.  In other 
words, there are situations when one has large incompatible
pre-strain, F\"oppl - von K\'arm\'an or milder bending.
A complete characterization of these theories remain an ongoing
endeavor.

Third, in Part C, we apply our results to liquid crystal glass where 
metric $G$ of the form $\mbox{Id}_3 +\gamma \vec n\otimes
\vec n$ for unitary director field  $\vec n\in\mathbb{R}^3$ and a
constant parameter $\gamma$. 
It has recently been suggested that such metrics and the resulting
deformation be exploited as means of actuation of micro-mechanical
devices \cite{MWB2,MWB}.  We show in  
Theorem \ref{lem71} that this metric is immersible 
if and only if $\mbox{curl}^T\mbox{curl}~ (\vec n \otimes \vec n) =
0$. Further, for the general three dimensional case 
we show in Theorem \ref{th3d} that the $\Gamma$-limit energy measures the bending content of the form:
$(\mbox{Id}_2 - \tilde\gamma n\otimes n) F_{2\times 2} (\mbox{Id}_2 -
\tilde\gamma n\otimes n)$, where $n$ is the in-plane component of $\vec n$
and $\tilde \gamma$ is an explicitly given inhomogeneous parameter.

\medskip

The paper is organized as follows.  Part A consists of three sections.  
We prove the lower bound to
$\frac{1}{h^2} E^h$ in section 2 and the upper bound in section 3.
section 4 specializes the formulas of the bending energy to the isotropic case.
Part B also consists of three sections.
In section 5, we derive equivalent conditions for the 
scaling $E^h\sim h^2$ to be optimal.  We provide examples of the
non-trivial (flat or non-flat) limiting configurations in section 6.  
In section 7 we show the equivalent condition for optimality of the
scaling $E^h\sim h^4$, in a particular case pertaining to one of the examples.
The application to nematic elastomers is in section 8 which is the sole section
in Part C.

Throughout the paper, we use the following notation.
Given a matrix $F\in \mathbb{R}^{3\times 3}$ we denote its trace by:
$\mbox{tr}~ F$, its transpose by: $F^T$, its symmetric part by: $\mbox{sym}~ F =
\frac{1}{2} (F + F^T)$, and its skew part by: $\mbox{skew}~F =
F-\mbox{sym}~F$.  We shall use the matrix norm $|F|= (\mbox{tr}(F^TF))^{1/2}$, which
is induced by the inner product $\langle F_1:F_2 \rangle = \mbox{tr}(F_1^TF_2)$.
The $k\times l$ principal minor of a matrix $F\in\mathbb{R}^{3\times
  3}$ will be denoted by $F_{k\times l}$. Conversely, for a given 
$F_{k\times l}\in\mathbb{R}^{k\times l}$, the $3\times 3$ matrix with
principal minor equal $F_{k\times l}$ and all other entries equal to $0$,
will be denoted $F_{k\times l}^*$.  
All limits are taken as the thickness parameter $h$
vanishes, i.e. when  $h\to 0$.  Finally, by $C$ we denote any universal constant, independent of $h$.

\medskip

\noindent{\bf Acknowledgments.} 
M.L. was partially supported by the NSF Career grant DMS-0846996 and
by the NSF grant DMS-1406730.
K.B. was partially supported by the NSF PIRE grant OISE-0967140.
\vspace{\baselineskip}

\begin{center}
{\bf Part A: Bending limit}
\end{center}

\section{The bending energy: the lower bound} 

We begin by characterizing sequences of deformations for which $E^h \le h^2$:
\begin{theorem}\label{thm2}
For a given sequence of deformations $u^h\in W^{1,2}(\Omega^h,\mathbb{R}^3)$ satisfying:
\begin{equation}\label{en_bound} 
E^h(u^h) \leq Ch^2,
\end{equation}
where $C$ is a uniform constant, there exists a sequence of translations
$c^h\in\mathbb{R}^3$, such that the following properties hold for the
normalised deformations $y^h\in W^{1,2}(\Omega^1,\mathbb{R}^3)$:
$$y^h(x', x_3) = u^h(x', hx_3) - c^h.$$
\begin{itemize}
\item[(i)] There exists $y\in W^{2,2}(\Omega,\mathbb{R}^3)$ such that,
  up to a subsequence:
$$ y^h\to y \qquad \mbox{ strongly in } W^{1,2}(\Omega^1,\mathbb{R}^3).$$
\item[(ii)] The deformation $y$ realizes the midplate metric: 
\begin{equation}\label{cinque}
(\nabla y)^T\nabla y = G_{2\times 2}.
\end{equation}
Consequently, the unit normal $\vec N$ to the surface $y(\Omega)$ and
the Cosserat vector $\vec b$ below have the regularity $\vec N, \vec
b\in W^{1,2}\cap L^\infty (\Omega,\mathbb{R}^3)$:
\begin{equation}\label{b}
{\vec N = \frac{\partial_1 y \times\partial_2y }{|\partial_1 y
    \times\partial_2y|}} \qquad \vec b = (\nabla y) (G_{2\times
  2})^{-1}\left[\begin{array}{c}G_{13}\\G_{23}\end{array}\right]
+ \frac{\sqrt{\det G}}{\sqrt{\det G_{2\times 2}}} \vec N.
\end{equation}
\item[(iii)] Define the quadratic forms:
\begin{equation}\label{Qform}
\begin{split}
 \mathcal{Q}_3(F) = D^2 W(\mathrm{Id})(F,F),\\
 \mathcal{Q}_2(x', F_{2\times 2}) = \min\left\{ \mathcal{Q}_3(\sqrt{G(x')}^{-1}\tilde F\sqrt{G(x')}^{-1});
  ~ \tilde F\in\mathbb{R}^{3\times 3} \mbox{ with }\tilde F_{2\times
    2} = F_{2\times 2}\right\}.
\end{split}
\end{equation}
The form $\mathcal{Q}_3$ is defined for  all $F\in\mathbb{R}^{3\times 3}$, 
while $\mathcal{Q}_2(x', \cdot)$ are defined on $F_{2\times
  2}\in\mathbb{R}^{2\times 2}$. Both forms $\mathcal{Q}_3$ and all $\mathcal{Q}_2$ are nonnegative
definite and depend only on the  symmetric parts of their arguments.
We then have the lower bound:
$$\liminf_{h\to 0} \frac{1}{h^2}E^h(u^h) \geq \mathcal{I}_G(y),$$
where:
\begin{equation}\label{IG}
\mathcal{I}_G(y) = \frac{1}{24}\int_\Omega \mathcal{Q}_2
\left(x', (\nabla y)^T \nabla \vec b \right)~\mathrm{d}x'.
\end{equation}
\end{itemize}
\end{theorem}

\begin{remark}
Consider a particular case, when the metric $G$ in (\ref{energy}) has the
structure $G=G_{2\times 2}^*+ e_3\otimes e_3$ as in \cite{sharon,lepa}.
Then, likewise: $A=A_{2\times 2}^* + e_3\otimes e_3$, and $A^{-1}e_3 = G^{-1}e_3 = e_3.$
From the formula (\ref{b}) it follows that: $\vec b = \vec N$, and so the asymptotic expansion of approximate
minimizers of (\ref{energy}) is: $u^h(x', x_3) \approx y(x') + x_3 \vec N(x')$.
Also, directly from (\ref{Qform}) we obtain:
\begin{equation*}
\begin{split}
\mathcal{Q}_2(x', F_{2\times 2}) &= \mathcal{Q}_2^0\Big(A_{2\times 2}^{-1}(x')
F_{2\times 2} A_{2\times 2}^{-1}(x')\Big) \\ & \mbox{where }
\mathcal{Q}_2^0(F_{2\times 2})=\min\left\{\mathcal{Q}_3(\tilde F); ~ \tilde
F\in\mathbb{R}^{3\times 3} \mbox{ with } \tilde F_{2\times 2} =
F_{2\times 2}\right\}.
\end{split}
\end{equation*}
Therefore, the limiting functional has the form:
$$\mathcal{I}_G(y) = \frac{1}{24}\int_\Omega \mathcal{Q}_2^0\Big(
A_{2\times 2}^{-1}(\nabla y)^T\nabla \vec N A_{2\times
  2}^{-1}\Big)~\mbox{d}x'$$
and it depends on $y$ only through the second fundamental form
$\Pi_y = (\nabla y)^T\nabla \vec N$ of the deformed mid-plate $y(\Omega)$.
For the isotropic density $W$ (see (\ref{Qiso})), one gets:
\begin{equation}\label{zero0}
\mathcal{Q}_2^0(F_{2\times 2}) = \mathcal{Q}_{2, iso}^0(F_{2\times 2}) = 
\mu|\mbox{sym} F_{2\times 2}|^2 + \frac{\lambda\mu}{\lambda+\mu}
|\mbox{tr}F_{2\times 2}|^2.
\end{equation}
We see that we recover the results of \cite{lepa} exactly.
\end{remark}

\medskip

Before proving Theorem \ref{thm2}, 
we first state the approximation lemma from \cite{lepa}, which is just
rephrasing Theorem 10 in \cite{FJMhier} in the present non-Euclidean
elasticity context.

\begin{lemma}\label{approx}
Assume (\ref{en_bound}). There exists matrix fields $Q^h\in
W^{1,2}(\Omega, \mathbb{R}^{3\times 3})$ such that:
\begin{itemize}
\item[(i)] $\displaystyle{\frac{1}{h} \int_{\Omega^h}|\nabla u^h(x', x_3) -
  Q^h(x')|^2 ~\mathrm{d}x \leq C \left( h^2 +
    \frac{1}{h}\int_{\Omega^h}\mathrm{dist}^2(\nabla u^hA^{-1}, SO(3))~\mathrm{d}x\right)}$,
\item[(ii)] $\displaystyle{\int_{\Omega}|\nabla  Q^h(x')|^2~\mathrm{d}x' \leq C 
\left(1 +  \frac{1}{h^3}\int_{\Omega^h}\mathrm{dist}^2(\nabla u^hA^{-1}, SO(3))~\mathrm{d}x\right)}$.
\end{itemize}
\end{lemma}

\medskip

\noindent {\bf Proof of Theorem \ref{thm2}.}

{\bf 1.} By Lemma \ref{approx} we see that the sequence $\{Q^h\}$ is
bounded in $L^{2}$, together with its derivatives. Therefore, up to a subsequence:
\begin{equation}\label{uno}
Q^h\rightharpoonup Q \quad \mbox{ weakly in } W^{1,2}(\Omega,
\mathbb{R}^{3\times 3}).
\end{equation}
Consider the rescaled deformations $y^h\in W^{1,2}(\Omega^1, \mathbb{R}^3)$ given by:
$$y^h(x', x_3) = u^h(x', hx_3) - \fint_{\Omega^h} u^h.  $$
Since:
$$\int_{\Omega^1} |\nabla u^h(x', hx_3) - Q(x')|^2 \leq 2\int_{\Omega^1}
|\nabla u^h(x', hx_3) - Q^h|^2 + 2\int_{\Omega} |Q^h - Q|^2,$$
it follows by Lemma \ref{approx} (i) and (\ref{uno}) that:
\begin{equation}\label{due} 
\left[\begin{array}{ccc} \partial_1y^h & \partial_2y^h &\displaystyle{\frac{\partial_3y^h}{h}}\end{array}\right]
\rightarrow Q\quad \mbox{ strongly in }
L^2(\Omega^1,\mathbb{R}^{3\times 3}).
\end{equation}
In particular, the sequence $\{\nabla y^h\}$ is bounded in $L^2$. Since $\fint y^h=0$, by the Poincar\'e
inequality, a subsequence of $\{y^h\}$ converges weakly in
$W^{1,2}(\Omega^1,\mathbb{R}^3)$ to some limiting field $y$. 
On the other hand, $\{\nabla y^h\}$ converge strongly because of
(\ref{due}):
$$\nabla_{tan}y^h\to Q_{3\times 2}  ~~~\mbox{ and } ~~~ \partial_3y^h
\to 0 \quad \mbox{ strongly in } L^2(\Omega^1).$$
Consequently, the convergence of $\{y^h\}$ is actually strong, and $y=y(x')\in
W^{2,2}(\Omega,\mathbb{R}^3)$ with:
\begin{equation}\label{tre}
\nabla y = \nabla_{tan}y = Q_{3\times 2}.
\end{equation}
We have thus proved (i) in Theorem \ref{thm2}. 

\medskip

{\bf 2.} Note that by Lemma \ref{approx} (i):
\begin{equation}\label{trea}
\begin{split}
\int_\Omega &\mbox{dist}^2(Q^hA^{-1}, SO(3))~\mbox{d}x' \\
& \leq \frac{C}{h} \left( \int_{\Omega^h}\mathrm{dist}^2(\nabla u^hA^{-1}, SO(3)) 
+ \int_{\Omega^h}|\nabla u^h(x', x_3)- Q^h(x')|^2~\mbox{d}x\right) \leq Ch^2.
\end{split}
\end{equation}
Therefore, by (\ref{uno}):
\begin{equation}\label{quattro}
QA^{-1}\in SO(3) \qquad \forall \mbox{a.e. } x'\in\Omega,
\end{equation}
so, in particular, we obtain (\ref{cinque}), and automatically:
\begin{equation}\label{reg}
\nabla y\in W^{1,2}\cap L^\infty(\Omega,\mathbb{R}^3).
\end{equation}
Further, by (\ref{cinque}) and using the formula $a\times (b\times c) = \langle a, c\rangle b
- \langle a, b\rangle c$, one gets:
\begin{equation}\label{crossp}
\begin{split}
|\partial_1y\times \partial_2y|^2 & = |Ae_1\times Ae_2|^2
= \langle Ae_1, (Ae_1\times Ae_2)\times
Ae_1\rangle \\ & = \langle Ae_2, \langle Ae_1, Ae_1\rangle
Ae_2 - \langle Ae_1, Ae_1\rangle
Ae_1\rangle \\ & =G_{11} G_{22} - G_{12}^2 = \det G_{2\times 2}.
\end{split}
\end{equation}
Hence, in view of (\ref{reg}): $\vec N\in W^{1,2}\cap L^\infty$ and,
consequently, the same holds for $\vec b$.

\medskip

{\bf 3.}  We will now prove that, assuming (\ref{tre}) and (\ref{cinque}), condition
(\ref{quattro}) is equivalent to $\vec b = Qe_3$ satisfty (\ref{b}).
Indeed, write:
$$\vec b = \alpha_1\partial_1 y + \alpha_2\partial_2y +
\alpha_3 \vec N.$$
By (\ref{crossp}), we obtain:
$$\det Q = \det \left[\begin{array}{ccc}\partial_1y & \partial_2y &
    \alpha_3\vec N \end{array}\right] =
\alpha_3|\partial_1y\times \partial_2y| = \alpha_3 \sqrt{\det G_{2\times 2}}.$$
Now, (\ref{quattro}) is equivalent to $Q^TQ=G$ and $\det Q>0$, hence
(\ref{quattro}) is further equivalent to:
\begin{equation*}
\begin{split}
G_{13} & =\langle \vec b, \partial_1y \rangle = \alpha_1 G_{11} + \alpha_2 G_{12}\\
G_{23} & =\langle\vec b, \partial_2y \rangle = \alpha_1 G_{21} + \alpha_2G_{22}\\
\sqrt{\det G} & = \det Q = \alpha_3\sqrt{\det G_{2\times 2}},
\end{split}
\end{equation*}
which yields:
$$\left[\begin{array}{c}\alpha_1\\ \alpha_2\end{array}\right]
= (G_{2\times 2})^{-1} \left[\begin{array}{c}G_{13}\\
    G_{23}\end{array}\right], \qquad 
\alpha_3 = \frac{\sqrt{\det G}}{\sqrt{\det G_{2\times 2}}},$$
exactly as claimed in (\ref{b}).

\medskip

{\bf 4.} We now modify the sequence $\{Q^h\}$ to another sequence $\tilde
Q^h\in L^2(\Omega, \mathbb{R}^{3\times 3})$ so that:
$$R^h = \tilde Q^hA^{-1}\in SO(3) \qquad \forall a.e. ~ x'\in\Omega,$$
This is done by projecting $\mathbb{P}_{SO(3)}$ onto $SO(3)$ when
possible, and  setting:
$$\tilde Q^hA^{-1} = \left\{
\begin{array}{ll} \mathbb{P}_{SO(3)}(Q^hA^{-1}) & \mbox{ if }
  Q^hA^{-1}\in \mathcal{O}_\epsilon (SO(3))\\
\mbox{Id} & \mbox{ otherwise}
\end{array}\right. $$
with a small $\epsilon > 0$. Then, by (\ref{trea}):
\begin{equation}\label{seia}
\int_\Omega |\tilde Q^h - Q^h|^2 \leq C\int_\Omega | \tilde
Q^hA^{-1} - Q^hA^{-1}|^2 \leq C \int_\Omega 
\mbox{dist}^2(Q^hA^{-1}, SO(3)) \leq Ch^2.
\end{equation}
In particular, by (\ref{uno}):
\begin{equation}\label{sei}
\tilde Q^h\to Q \qquad \mbox{ strongly in } L^2(\Omega,
\mathbb{R}^{3\times 3}).
\end{equation}

Define the scaled strains $S^h\in L^2(\Omega^1, \mathbb{R}^{3\times 3})$ by:
$$S^h(x', x_3) = \frac{1}{h} \left( (R^h)^T\nabla u^h(x', hx_3)
  A^{-1} - \mbox{Id}\right).$$
We have, in view of Lemma \ref{approx} (i) and (\ref{seia}):
\begin{equation}\label{seib}
\begin{split}
\int_{\Omega^1} |S^h|^2  & \leq \frac{C}{h^2} \int_{\Omega^1} |\nabla
u^h(x', hx_3) - \tilde Q^h|^2~\mbox{d}x \\ & \leq
\frac{C}{h^3}\int_{\Omega^h} |\nabla u^h - Q^h|^2 +
\frac{C}{h^2}\int_\Omega |Q^h - \tilde Q^h|^2 \leq C,
\end{split}
\end{equation}
and hence a subsequence of $\{S^h\}$ converges:
\begin{equation}\label{sette}
S^h\rightharpoonup \bar S \qquad \mbox{ weakly in } L^2(\Omega^1,
\mathbb{R}^{3\times 3}).
\end{equation}

\medskip

{\bf 5.} We now derive the formula on the limiting strain $\bar S$. 
Consider the difference quotients:
$$f^{s,h}(x', x_3) = \frac{1}{h} \frac{1}{s} (y^h(x', x_3+s) - y^h(x',
x_3)) \in L^2(\Omega^1, \mathbb{R}^{3\times 3}).$$
By (\ref{due}), it follows that:
$$f^{s,h}(x', x_3) = \frac{1}{h}\fint_0^s\partial_3y^h(x',
x_3+t)~\mbox{d}t \rightarrow \vec b(x') \qquad \mbox{ in }
L^2(\Omega^1, \mathbb{R}^{3}).$$
Similarly:
$$\partial_3 f^{s,h}(x', x_3) = \frac{1}{s}
\left(h^{-1}\partial_3y^h(x', x_3+s) - h^{-1}\partial_3y^h(x', x_3)
\right) \rightarrow 0 \qquad \mbox{ strongly in }
L^2(\Omega^1, \mathbb{R}^{3}),$$
while for $i=1,2$, by (\ref{sei}) and (\ref{sette}):
\begin{equation*}
\begin{split}
\partial_i f^{s,h}(x', x_3) & = \frac{1}{h}\frac{1}{s} \left(\nabla
  u^h(x', h(x_3+s)) - \nabla u^h(x', hx_3)\right)e_i \\ & = \frac{1}{s}
R^h(x') \left(S^h(x', x_3+s) - S^h(x', x_3)\right) Ae_i \\ & \qquad 
\rightharpoonup \frac{1}{s} Q A^{-1} \left(\bar S(x', x_3+s) -
  \bar S(x', x_3)\right) Ae_i \qquad \mbox{ weakly in } L^2(\Omega^1,\mathbb{R}^3). 
\end{split}
\end{equation*}
Concluding:
$$f^{s,h}\rightharpoonup \vec b 
\qquad \mbox{ weakly in } W^{1,2}(\Omega^1,\mathbb{R}^3),$$
and hence:
$$\forall i=1,2 \qquad 
\partial_i\vec b (x') = \frac{1}{s} Q A^{-1} \left(\bar S(x', x_3+s) -
  \bar S(x', x_3)\right) Ae_i.$$
By (\ref{quattro}), $QA^{-1}$ may be replaced by $Q^{T,
  -1}A$, so that:
$$\forall i=1,2 \qquad \bar S(x', x_3+s) Ae_i  = \bar S(x', x_3) Ae_i +
sA^{-1} Q^T \partial_i\vec b,$$
and in view of (\ref{tre}) we obtain:
\begin{equation}\label{otto}
\left(A~\bar S(x', x_3+s)A\right)_{2\times 2} 
= \left(A\bar S(x', x_3)A\right)_{2\times 2} + s(\nabla
y)^T\nabla \vec b.
\end{equation}

\medskip

{\bf 6.} We now compute the lower bound on the rescaled energies. 
Define the 'good' sets:
$$\Omega^1_h=\{(x', x_3)\in\Omega^1; ~ |S^h(x', x_3)|^2\leq \frac{1}{h} \}.$$
In view of (\ref{seib}), it follows the convergence of characteristic functions:
\begin{equation*}
\chi_h = \chi_{\Omega^1_h} \rightarrow 1 \qquad \mbox{ strongly in } L^1(\Omega^1).
\end{equation*}
and therefore, by (\ref{sette}):
\begin{equation}\label{nove}
\chi_h S^h \rightharpoonup \bar S \qquad \mbox{ weakly  in } L^2(\Omega^1,
\mathbb{R}^{3\times 3}).
\end{equation}
For small $h$, we may Taylor expand $W$ on the 'good' sets, using the definition of
$S^h$:
\begin{equation*}
\begin{split}
\forall (x', x_3)\in \Omega^1_h \qquad \frac{1}{h^2} W\left(\nabla u^h(x',
  hx_3)A^{-1}\right) & = \frac{1}{h^2} W(\mbox{Id} + hS^h(x',
x_3)) \\ & = \frac{1}{2} \mathcal{Q}_3(S^h(x', x_3)) + o(|S^h|^2).
\end{split}
\end{equation*}
By (\ref{nove}), we now obtain:
\begin{equation}\label{dieci}
\begin{split}
\liminf_{h\to 0} \frac{1}{h^2}E^h(u^h)  & \geq \liminf_{h\to 0}
\frac{1}{h^2}\int_{\Omega_h^1} W\left(\nabla u^h(x',
  hx_3)A^{-1}\right)~\mbox{d}x \\ & = 
\liminf_{h\to 0} \frac{1}{2} \int_{\Omega^1}\mathcal{Q}_3\left(\chi_h S^h(x',
  x_3)\right) \geq \frac{1}{2} \int_{\Omega^1}\mathcal{Q}_3\left(\bar
  S\right). 
\end{split}
\end{equation}
Since the quadratic form $\mathcal{Q}_3$ is nonnegative definite, we obtain:
\begin{equation*}
\begin{split}
\frac{1}{2} \int_{\Omega^1}\mathcal{Q}_3\left(\bar
  S\right) & \geq 
\frac{1}{2} \int_{\Omega^1}\mathcal{Q}_2\left((A ~\bar S(x',
  x_3) A)_{2\times 2}\right) \\ &
= \frac{1}{2} \int_{\Omega}\mathcal{Q}_2\left((A ~\bar S(x',0)
  A)_{2\times 2}\right)~\mbox{d}x' + \frac{1}{2}
\left(\int_{-1/2}^{1/2} s^2~\mbox{d}s\right)
\int_{\Omega}\mathcal{Q}_2\left(\nabla y)^T \nabla \vec
  b\right)~\mbox{d}x' \\ & 
\geq \frac{1}{24} \int_{\Omega}\mathcal{Q}_2\left(\nabla y)^T \nabla \vec
  b\right)~\mbox{d}x' = \mathcal{I}_G(y),
\end{split}
\end{equation*}
where we used (\ref{otto}). In view of (\ref{dieci}), the proof is complete.
\endproof

\section{The bending energy: recovery sequence and the upper bound}

In this section we prove that the lower bound in Theorem \ref{thm2} is
optimal, in the following sense:

\begin{theorem}\label{thm3}
For every isometric immersion $y\in W^{2,2}(\Omega,\mathbb{R}^3)$  of
the metric $ G_{2\times 2}$ as in (\ref{cinque}),  there exists a
sequence of 'recovery deformations' $u^h\in
W^{1,2}(\Omega^h,\mathbb{R}^3)$, such that:
\begin{itemize}
\item[(i)]  The rescaled sequence $y^h(x', x_3) = u^h(x', hx_3)$ converges in
$W^{1,2}(\Omega^1,\mathbb{R}^3)$ to $y$.
\item[(ii)] One has: $$\lim_{h\to 0} \frac{1}{h^2}E^h(u^h) = \mathcal{I}_G(y),$$
where the Cosserat vector $\vec b$ in the definition (\ref{IG}) of the
functional $\mathcal{I}_G$ is derived by (\ref{b}).
\end{itemize}
\end{theorem}

It immediately follows that:

\begin{corollary}\label{upper}
Existence of a $W^{2,2}$ regular isometric immersion of the Riemannian
metric $G_{2\times 2}$ on $\Omega$ in $\mathbb{R}^3$ is equivalent to the upper bound on the energy
scaling at minimizers:
$$\exists C>0 \qquad \inf_{u\in W^{1,2}(\Omega^h,\mathbb{R}^3)} E^h(u)
\leq Ch^2.$$
\end{corollary}

\begin{corollary}\label{infmin}
The limiting functional $\mathcal{I}_G$ attains its minimum.
\end{corollary}
\begin{proof}
Let $\{y_n\}_{n=1}^\infty$ be a minimizing sequence of
$\mathcal{I}_G$. By Theorem \ref{thm3}, there exists sequences
$u_n^h\in W^{1,2}(\Omega^h,\mathbb{R}^3)$ such that:
$\lim_{h\to 0} u^h_n(x', hx_3) = y_n$ in
$W^{1,2}(\Omega^1,\mathbb{R}^3)$ and $\lim_{h\to 0} \frac{1}{h^2}
E^h(u_n^h)=\mathcal{I}_G(y_n)$, for every $n$. Taking $u^h=u_n^{h(n)}$
for a sequence $h(n)$ converging to $0$ as $n\to\infty$ sufficiently
fast, we obtain: $E^h(u^h)\leq Ch^2$. Therefore, by Theorem
\ref{thm2} there exists a limiting deformation $y\in
W^{2,2}(\Omega,\mathbb{R}^3)$ so that:
$$\mathcal{I}_G(y) \leq \liminf_{h\to 0} \frac{1}{h^2} E^h(u^h) = \lim_{n\to\infty}
\mathcal{I}_G(y_n) = \inf \mathcal{I}_G,$$
which achieves that $y$ is a minimizer of $\mathcal{I}_G$.
\end{proof}

\bigskip

Before proving Theorem \ref{thm3},  recall that:
\begin{equation}\label{cmin}
\begin{split}
\forall F_{2\times 2}\in\mathbb{R}^{2\times 2}_{sym}\qquad 
\mathcal{Q}_2(x', F_{2\times 2}) & = \min\left\{ \mathcal{Q}_3(A^{-1}\tilde
  F A^{-1}); ~ \tilde F\in\mathbb{R}^{3\times 3}, ~ \tilde F_{2\times 2} =
  F_{2\times 2}\right\}\\ 
& = \min\left\{ \mathcal{Q}_3\big(A^{-1}(F^*_{2\times 2} +
  \mbox{sym}(c\otimes e_3)) A^{-1}\big); ~ c\in\mathbb{R}^3\right\}.
\end{split}
\end{equation}
In what follows, by:
$$c(x', F_{2\times 2})$$
we will denote the unique minimizer of the problem in (\ref{cmin}).

\medskip

\noindent {\bf Proof of Theorem \ref{thm3}.}

{\bf 1.} Let $y\in W^{2,2}(\Omega, \mathbb{R}^3)$ satisfy
(\ref{cinque}). Define the Cosserat vector field $\vec b\in
W^{1,2}\cap L^\infty(\Omega, \mathbb{R}^3)$ according to (\ref{b}) and
let: 
$$Q= \left[\begin{array}{ccc} \partial_1y &\partial_2 y & \vec
    b\end{array}\right] \in W^{1,2}\cap L^\infty(\Omega, \mathbb{R}^{3\times 3}).$$
By Step 2 in the proof of Theorem \ref{thm2}, it follows that:
\begin{equation}\label{jeden}
QA^{-1}\in SO(3) \qquad \forall \mbox{a.e.}~ x'\in\Omega.
\end{equation}
Define the limiting warping field $\vec d\in L^2(\Omega, \mathbb{R}^3)$:
\begin{equation}\label{d_warp}
\vec d(x') = Q^{T, -1}\left( c\big(x', (\nabla y)^T\nabla \vec b\big) -
  \frac{1}{2}\nabla |\vec b|^2\right).
\end{equation}
Let $\{d^h\}$ be a approximating sequence in $W^{1,\infty}(\Omega, \mathbb{R}^3)$, satisfying:
\begin{equation}\label{dh}
d^h\to \vec d \quad \mbox{ strongly in } L^2(\Omega, \mathbb{R}^3), \quad \mbox{ and}\quad
h\|d^h\|_{W^{1, \infty}} \to 0.
\end{equation}
Note that such sequence can always be derived by reparametrizing
(slowing down) a sequence of smooth approximations of $\vec d$.
Similiarly, consider the approximations $y^h\in W^{2,\infty}(\Omega,
\mathbb{R}^3)$ and $\vec b^h\in W^{1,\infty}(\Omega,\mathbb{R}^3)$,
with the following properties:
\begin{equation}\label{dwa}
\begin{split}
& y^h\to y \quad \mbox{ strongly in } W^{2,2}(\Omega,\mathbb{R}^3),
\quad \mbox{and } ~~ \vec b^h\to \vec b \quad \mbox{ strongly in }
W^{1,2}(\Omega,\mathbb{R}^3)\\
& h\left( \|y^h\|_{W^{2,\infty}} + \|\vec
  b^h\|_{W^{1,\infty}}\right)\leq \epsilon\\
&\frac{1}{h^2}|\Omega\setminus \Omega_h| \to 0, \quad \mbox{where }~~
\Omega_h =\left\{x'\in\Omega; ~ y^h(x') =  y(x') \mbox{ and } \vec b^h(x') = \vec b(x')\right\} 
\end{split}
\end{equation}
for some small $\epsilon> 0$.
Existence of approximations with the claimed properties follows by
partition of unity and truncation arguments, as a special case of the
Lusin-type result for Sobolev functions in \cite{liu} (see also
{Proposition 2 in \cite{FJMhier})}.

We now define $u^h\in W^{1,\infty}(\Omega^h, \mathbb{R}^3)$ by:
\begin{equation*}
u^h(x', x_3) = y^h(x') + x_3 \vec b^h(x') + \frac{x_3^2}{2}d^h(x').
\end{equation*}
Consequently, the rescalings $y^h\in W^{1,\infty}(\Omega^1,
\mathbb{R}^3)$ are:
$$ y^h(x', x_3) = y^h(x') + hx_3 \vec b^h(x') + \frac{h^2}{2}x_3^2 d^h(x'),$$
and therefore in view of (\ref{dh}) and (\ref{dwa}), Theorem
\ref{thm3} (i) follows.:

\medskip

{\bf 2.} Define the matrix fields:
$$Q^h(x') = \left[\begin{array}{ccc}\partial_1 y^h & \partial_2 y^h&
    \vec b^h \end{array}\right], \quad 
B^h(x') = \left[\begin{array}{ccc}\partial_1 \vec b^h & \partial_2
    \vec b^h & d^h \end{array}\right], \quad
D^h(x') = \left[\begin{array}{ccc}\partial_1 d^h &\partial_2 d^h&
   0 \end{array}\right], $$
so that:
$$ \nabla u^h(x', x_3) = Q^h(x') + x_3 B^h(x') + \frac{x_3^2}{2}
D^h(x') \qquad \forall (x', x_3)\in \Omega^h. $$
Since $Q^h=Q$ in the set $\Omega_h$, then by (\ref{jeden}) and the bound on
the Lipschitz constants of $y^h$ and $\vec b^h$ in (\ref{dwa}), we
obtain:
\begin{equation}\label{trzy}
\mbox{dist}(Q^hA^{-1}, SO(3)) \leq \frac{C}{h}
\mbox{dist}(x', \Omega_h)\leq \frac{C}{h} |\Omega\setminus
\Omega_h|^{1/2}.
\end{equation}
The last bound above can be easily obtained by noting that if
$B_r(x')\subset \Omega\setminus \Omega_h$ then $\pi r^2\leq
|\Omega\setminus \Omega_h|$, which implies $r\leq C
|\Omega\setminus \Omega_h|^{1/2}$. For $x'$ close to the boundary of
$\Omega$ one needs to slightly refine the argument using smoothness of
$\partial\Omega$. 

Consequently, by (\ref{trzy}) and (\ref{dwa}), it follows that for all $h$ sufficiently small:
\begin{equation*}
\begin{split}
\mbox{dist} \Big(\nabla u^h(x', hx_3) &A^{-1}, SO(3)\Big) \leq 
\mbox{dist}(Q^hA^{-1}, SO(3)) + h\|B^h\|_{L^{\infty}} + h^2\|D^h\|_{L^\infty}\\
& \leq \frac{C}{h} |\Omega\setminus
\Omega_h|^{1/2} + Ch(\|\nabla\vec b^h\|_{L^\infty} +
\|d^h\|_{L^\infty}) + Ch^2 \|\nabla d^h\|_{L^\infty}\leq \epsilon_0,
\end{split}
\end{equation*}
where $\epsilon_0$ is such that the energy density $W$ is bounded and 
$\mathcal{C}^2$ regular in the neighbourhood $\mathcal{O}_{\epsilon_0}(SO(3))$.
Taylor expanding $W$ at the given rotation in (\ref{jeden}), we compute:
\begin{equation*}
\begin{split}
\frac{1}{h^2}&\int_{\Omega_h\times (-1/2, 1/2)}  W\Big(\nabla u^h(x', hx_3)A^{-1}\Big)\\
& = \frac{1}{h^2}\int_{\Omega_h\times (-1/2, 1/2)} W\left(( Q(x') + hx_3 B^h(x') +
  h^2\frac{x_3^2}{2} D^h(x'))A^{-1}\right)~\mbox{d}x\\
& = \frac{1}{2} \int_{\Omega_h\times (-1/2, 1/2)} D^2W(Q(x')A^{-1}) \left((x_3 B^h(x') +
  h\frac{x_3^2}{2} D^h(x'))A^{-1}\right)^{\otimes 2} +\mathcal{O}(h)~\mbox{d}x.
\end{split}
\end{equation*}
Also, by (\ref{dwa}):
\begin{equation*}
\frac{1}{h^2}\int_{(\Omega\setminus\Omega_h)\times (-1/2, 1/2)} W\Big(\nabla u^h(x', hx_3)A^{-1}\Big)
\leq \frac{C}{h^2}|\Omega\setminus \Omega_h| \quad \to 0.
\end{equation*}
Hence:
\begin{equation}\label{lim}
\begin{split}
\lim_{h\to 0}&\frac{1}{h^2}E^h(u^h) = 
\lim_{h\to 0}\frac{1}{h^2}\int_{\Omega_h\times (-1/2, 1/2)} W\Big(\nabla u^h(x',
hx_3)A^{-1}\Big) \\ & =  \lim_{h\to 0} \frac{1}{2}\int_{\Omega_h\times (-1/2, 1/2)} 
D^2W(QA^{-1}) \left((x_3 B^h(x') +  h\frac{x_3^2}{2} D^h(x'))A^{-1}\right)^{\otimes 2}
~\mbox{d}x \\ & = \lim_{h\to 0} \frac{1}{2}\int_{-1/2}^{1/2}\int_{\Omega_h} x_3^2
D^2W(QA^{-1}) \left(B^h(x') A^{-1}\right)^{\otimes 2}
~\mbox{d}x'~\mbox{d}x_3 \\ &  = \lim_{h\to 0}\frac{1}{24}\int_{\Omega_h}
\mathcal{Q}_3\left((QA^{-1})^T B^h(x') A^{-1}\right) \\
&  = \frac{1}{24}\int_{\Omega} \mathcal{Q}_3\left(A^{-1} Q^T B A^{-1}\right),
\end{split}
\end{equation}
where we have used the last convergence in (\ref{dwa}), the frame
invariance of the density function $W$ resulting in: $D^2W(R)(F,F) = D^2W(\mbox{Id})(R^TF,
R^TF) = \mathcal{Q}_3(R^TF)$ valid for all $R\in SO(3)$, and the
following convergence:
$$B^h\to B(x') = \left[\begin{array}{ccc}\partial_1 \vec b & \partial_2
    \vec b & \vec d \end{array}\right] \qquad \mbox{strongly in }
L^2(\Omega,\mathbb{R}^{3\times 3}).$$
Now, note that by(\ref{d_warp}):
\begin{equation*}
\begin{split}
\mbox{sym} (Q^TB(x')) & = \mbox{sym}((\nabla y)^T\nabla\vec b)
+\mbox{sym}(e_3\otimes \left[\begin{array}{c} (\nabla y)^T\vec d +
    \frac{1}{2}\nabla |\vec b|^2 \\ 
\langle\vec b, \vec d\rangle \end{array}\right])\\ &
= \mbox{sym}((\nabla y)^T\nabla\vec b)
+\mbox{sym}\Big(e_3\otimes c(x', (\nabla y)^T\nabla\vec b)\Big).
\end{split}
\end{equation*}
Therefore, (\ref{lim}) becomes:
$$\lim_{h\to 0}\frac{1}{h^2}E^h(u^h) = \frac{1}{24}\int_{\Omega}
\mathcal{Q}_2\left((\nabla y)^T\nabla \vec b\right)~\mbox{d}x',$$
achieving the proof of Theorem \ref{thm3}.
\endproof

\section{The effective density $\mathcal{Q}_2$ and the case of $W$ isotropic}

In this section, we further study the 2d functional (\ref{IG}) and
the inhomogeneous effective energy measure in (\ref{cmin}).
By $L_3:\mathbb{R}^{3\times 3}\to\mathbb{R}^{3\times 3}$ we denote the
linear map with the property that:
$$\mathcal{Q}_3(F) = \langle L_3(F) : F\rangle \quad \mbox{ and }
\quad \langle L_3(F) : \tilde F\rangle = \langle L_3(\tilde F) : F\rangle
\qquad \forall F, \tilde F\in \mathbb{R}^{3\times 3}.$$
Note that by frame invariace of $W$ in (\ref{cond}) one has: $L_3(F) =
L_3(\mbox{sym}F)$ and $\mbox{skew}(L_3(F)) = 0$.

\begin{lemma}
Define the matrix field $M_A:\Omega\to\mathbb{R}^{3\times 3}$ by:
$$ \forall i:1\ldots 3 \qquad M_Ae_i = L_3(e_i\otimes A^{-1}e_3) A^{-1}e_3.$$
Then the unique minimizer $c_0 = c(x', F_{2\times 2})$ in (\ref{cmin})
is given by:
\begin{equation}\label{c0}
A^{-1}c_0 = - M_A^{-1} L_3(A^{-1}F_{2\times 2}^* A^{-1}) A^{-1}e_3.
\end{equation}
Consequently:
\begin{equation}\label{Qform2}
\mathcal{Q}_2(F_{2\times 2}) = \mathcal{Q}_3(A^{-1}F_{2\times 2}^*
A^{-1}) - \Big\langle M_A^{-1}L_3(A^{-1}F_{2\times 2}^* A^{-1})
A^{-1}e_3, L_3(A^{-1}F_{2\times 2}^* A^{-1}) A^{-1}e_3\Big\rangle,
\end{equation}
\end{lemma}
\begin{proof}
For $i:1..3$ we have:
\begin{equation*}
\begin{split}
\frac{\mbox{d}}{\mbox{d}c_i}\mathcal{Q}_3(A^{-1}(F^*_{2\times 2} +
c\otimes e_3) A^{-1}) & = 2\Big\langle L_3 (A^{-1}(F^*_{2\times 2} +
c\otimes e_3) A^{-1}) : A^{-1} e_i\otimes A^{-1}e_3\Big\rangle \\ & = 
2\Big\langle A^{-1}L_3 \left( A^{-1}(F^*_{2\times 2} +
c\otimes e_3) A^{-1}\right)A^{-1} : e_i\otimes e_3\Big\rangle.
\end{split}
\end{equation*}
Therefore, at the miminizer $c_0$ we have:
\begin{equation*}
\begin{split}
\nabla_c \mathcal{Q}_3 (A^{-1}(F^*_{2\times 2} + c_0\otimes e_3)
A^{-1}) & = 2A^{-1}L_3\left(A^{-1}(F^*_{2\times 2} +
c_0\otimes e_3) A^{-1}\right)A^{-1}e_3  \\ & = 2A^{-1}L_3( A^{-1}F^*_{2\times 2}A^{-1} + A^{-1}c_0 \otimes
A^{-1 }e_3)A^{-1}e_3 = 0,
\end{split}
\end{equation*}
which is equivalent to:
$$-L_3(A^{-1}F^*_{2\times 2} A^{-1}) A^{-1}e_3 = L_3(A^{-1}c_0 \otimes
A^{-1 }e_3) A^{-1}e_3 = M_A A^{-1}c_0,$$
and consequently to (\ref{c0}). Then:
\begin{equation}
\begin{split}
\mathcal{Q}_2&(F_{2\times 2})  = \mathcal{Q}_3(A^{-1}F_{2\times 2}^*A^{-1} +
A^{-1}c_0\otimes A^{-1}e_3) \\ & = 
\Big\langle L_3(A^{-1}F_{2\times 2}^*A^{-1}) + L_3(A^{-1}c_0 \otimes A^{-1
}e_3) : A^{-1}F_{2\times 2}^*A^{-1}\Big\rangle \\ & 
= \Big\langle L_3(A^{-1}F_{2\times 2}^*A^{-1}) : A^{-1}F_{2\times
  2}^*A^{-1}  + A^{-1}c_0 \otimes A^{-1}e_3 \Big\rangle \\ & 
= \mathcal{Q}_3(A^{-1}F_{2\times 2}^*
A^{-1}) - \Big\langle L_3(A^{-1}F_{2\times 2}^* A^{-1}) :
M_A^{-1}L_3(A^{-1}F_{2\times 2}^* A^{-1}) (A^{-1}e_3\otimes A^{-1}e_3)\Big\rangle,
\end{split}
\end{equation}
which proves (\ref{Qform2}).
\end{proof}

\medskip

We now assume that the energy density $W$ is isotropic, i.e.:
$$\forall F\in\mathbb{R}^{3\times 3}\quad\forall R\in SO(3)\qquad
W(RF) = W(F).$$
It is known \cite{fried_book} (see also \cite{FJMhier}  and 
Appendix A in \cite{BLZ}) that $\mathcal{Q}_3$ is then given in
terms of the Lam\'e coefficients $\lambda, \mu$:
\begin{equation}\label{Qiso}
\mathcal{Q}_3(F) = \mu |\mathrm{sym} F|^2 + \lambda |\mathrm{tr} F|^2,
\end{equation}
and so we also have:
\begin{equation}\label{zero}
L_3(F) = \mu~\mbox{sym} F + \lambda (\mbox{tr} F)~\mbox{Id}.
\end{equation}

\begin{lemma}\label{isotropic}
Assume that $W$ is isotropic, so that (\ref{Qiso}) holds. Then:
\begin{equation}\label{MA}
M_A = \frac{\mu}{2} |A^{-1}e_3|^2\mathrm{Id} + (\lambda +
\frac{\mu}{2}) (A^{-1}e_3 \otimes A^{-1}e_3)
\end{equation}
and, denoting $D= A^{-1}F_{2\times 2}^*A^{-1}$ and $d= A^{-1}e_3$, we have:
\begin{equation}\label{iso2}
\forall F_{2\times 2}\in\mathbb{R}^{2\times 2}_{sym} \qquad
\mathcal{Q}_2(x', F_{2\times 2}) 
 = \mu\left(|D|^2 - 2\frac{|Dd|^2}{|d|^2} 
+\frac{\langle Dd, d\rangle^2}{|d|^4}\right) + 
\frac{\lambda\mu}{\lambda+\mu}\left(\mathrm{tr} D  - \frac{\langle Dd,
    d\rangle}{|d|^2}\right)^2.  
\end{equation}
\end{lemma}
\begin{proof}
By (\ref{zero}), we obtain:
$$M_Ae_i = L_3(e_i\otimes d) d
= (\lambda + \frac{\mu}{2})\langle d, e_i\rangle d +
\frac{\mu}{2}|d|^2 e_i$$
which gives (\ref{MA}).
It is easy to check directly the following general formula:
$$(\alpha~\mbox{Id} + a\otimes b)^{-1} = \frac{1}{\alpha}\mbox{Id} -
\frac{1}{\alpha(\alpha+\langle a, b\rangle)} a\otimes b.$$
Applying it to $\alpha = \frac{\mu}{2}|d|^2$ and $a=(\lambda
+\frac{\mu}{2})d$ and $b=d$, we get:
$$ M_A^{-1}= \frac{2}{\mu} \frac{1}{|d|^2} \mbox{Id} - \frac{2\lambda
  + \mu}{\mu(\lambda +\mu)} \frac{1}{|d|^4} (d\otimes d).$$
Therefore:
\begin{equation*}
\langle M_A^{-1} L_3(D)d, L_3(D)d\rangle =
\frac{\lambda^2}{\lambda+\mu}(\mbox{tr} D)^2 +
2\frac{\lambda\mu}{\lambda + \mu}(\mbox{tr} D)\frac{\langle Dd,
  d\rangle}{|d|^2}  + 2\mu \frac{|Dd|^2}{|d|^2} - \frac{(2\lambda +
  \mu)\mu}{\lambda+\mu}\frac{\langle Dd, d\rangle^2}{|d|^4}.
\end{equation*}
Concluding:
\begin{equation*}
\begin{split}
\mathcal{Q}_2(x', F_{2\times 2}) & = \frac{\lambda\mu}{\lambda+\mu}(\mbox{tr} D)^2 +
\mu|D|^2 - 2\frac{\lambda\mu}{\lambda + \mu}(\mbox{tr} D)\frac{\langle Dd,
  d\rangle}{|d|^2}  - 2\mu \frac{|Dd|^2}{|d|^2} + \frac{(2\lambda +
  \mu)\mu}{\lambda+\mu}\frac{\langle Dd, d\rangle^2}{|d|^4}
\end{split}
\end{equation*}
which yields (\ref{iso2}).
\end{proof}

\begin{theorem}\label{isobest}
Assume that $W$ is isotropic, so that (\ref{Qiso}) holds. Then:
\begin{equation}\label{iso3}
\begin{split}
\forall F_{2\times 2}\in\mathbb{R}^{2\times 2}_{sym} \qquad
& \mathcal{Q}_2(x', F_{2\times 2})
= \mathcal{Q}_{2, iso}^0\left(\sqrt{G_{2\times 2}}^{-1}F_{2\times
    2}\sqrt{G_{2\times 2}}^{-1}\right) \\ 
& = \mu\left| \sqrt{G_{2\times 2}}^{-1}F_{2\times 2}\sqrt{G_{2\times 2}}^{-1}\right|^2 
+ \frac{\lambda\mu}{\lambda+\mu}
|\mathrm{tr}\left(\sqrt{G_{2\times 2}}^{-1}F_{2\times
    2}\sqrt{G_{2\times 2}}^{-1}\right)|^2 
\end{split}
\end{equation}
\end{theorem}
\begin{proof}
Given $v\in\mathbb{R}^3$, we denote $v_{tan}=(v_1, v_2)^T\in\mathbb{R}^2$. 
As in the proof of Theorem \ref{th12}, given $F\in\mathbb{R}^{3\times 3}$, by
$F_{tan}\in\mathbb{R}^{2\times 2}$ we denote the principal $2\times 2$
minor of $F$, and we let $F_{cross}= (Fe_3)_{tan} = (F_{13},
F_{23})^T \in\mathbb{R}^2$. 
We now use the notation of Lemma \ref{isotropic} and identify the terms in
(\ref{iso2}). Call $P=G^{-1}$. Then:
\begin{equation*}
\begin{split}
|D|^2 & = \left\langle PF^*_{2\times 2}P : F^*_{2\times 2}\right\rangle
= \left\langle (PF^*_{2\times 2}P)_{tan}: F_{2\times 2}\right\rangle =
\left\langle P_{tan}F_{2\times 2}P_{tan} :F_{2\times 2}\right\rangle\\
|Dd|^2 & = \left\langle PF^*_{2\times 2}Pe_3, F^*_{2\times 2}Pe_3\right\rangle
= \left\langle (PF^*_{2\times 2}Pe_3)_{tan}, F_{2\times
    2}P_{cross}\right\rangle = \left\langle P_{tan}F_{2\times 2}P_{cross}, F_{2\times
    2}P_{cross}\right\rangle\\
\langle Dd, d\rangle & = \langle PF_{2\times 2}^* Pe_3, e_3\rangle =
\langle FP_{cross}, P_{cross}\rangle\\
|d|^2 & = \left\langle Pe_3, e_3\right\rangle = P_{33}\\
\mbox{tr} D & = \mbox{tr} (PF^*_{2\times 2}) = \mbox{tr} (P_{tan}F_{2\times 2}).
\end{split}
\end{equation*}
Hence, (\ref{iso2}) becomes:
\begin{equation}\label{3}
\begin{split}
\mathcal{Q}_2(x', F_{2\times 2}) 
= & \mu\left( \left\langle P_{tan}F_{2\times 2}P_{tan} :F_{2\times
      2}\right\rangle - 2\frac{\left\langle P_{tan}F_{2\times 2}P_{cross}, F_{2\times
    2}P_{cross}\right\rangle}{P_{33}} + \frac{\langle FP_{cross}, P_{cross}\rangle^2}{(P_{33})^2}
\right) \\ 
& + \frac{\lambda\mu}{\lambda+\mu}
\left(\mbox{tr}(P_{tan}F_{2\times 2}) - \frac{\langle FP_{cross}, P_{cross}\rangle}{P_{33}}\right)^2.
\end{split}
\end{equation}

We now identify the terms in the right hand side of (\ref{iso3}),
using the formulas (\ref{4}) and (\ref{44}):
\begin{equation*}
\begin{split}
|\sqrt{G_{tan}}^{-1} & F_{2\times 2} \sqrt{G_{tan}}^{-1}|^2 = \langle
(G_{tan})^{-1}F_{2\times 2} (G_{tan})^{-1} : F_{2\times 2} \rangle \\
& = \langle (P_{tan} - \frac{1}{P_{33}}P_{cross}\otimes P_{cross})
F_{2\times 2} (P_{tan} - \frac{1}{P_{33}}P_{cross}\otimes P_{cross}) : 
 F_{2\times 2} \rangle \\  & = \left\langle P_{tan}F_{2\times 2}P_{tan} :F_{2\times
      2}\right\rangle - \frac{2}{P_{33}} \langle (P_{cross}\otimes
  P_{cross}) F P_{tan} : F_{2\times 2}\rangle \\ & \qquad + \frac{1}{(P_{33})^2}
  \langle (P_{cross}\otimes P_{cross}) F (P_{cross}\otimes
  P_{cross}) : F\rangle \\ & = \left\langle P_{tan}F_{2\times 2}P_{tan} :F_{2\times
      2}\right\rangle - 2\frac{\left\langle P_{tan}F_{2\times 2}P_{cross}, F_{2\times
    2}P_{cross}\right\rangle}{P_{33}} + \frac{\langle FP_{cross},
P_{cross}\rangle^2}{(P_{33})^2},\\
\mbox{tr}(\sqrt{G_{tan}}&^{-1} F_{2\times 2}
  \sqrt{G_{tan}}^{-1} ) = \mbox{tr}\left((G_{tan})^{-1} F_{2\times 2}\right)
 = \mbox{tr}\left( (P_{tan} - \frac{1}{P_{33}}P_{cross}\otimes P_{cross})
F_{2\times 2}\right) \\ & 
= \mbox{tr}(P_{tan}F_{2\times 2}) - \frac{1}{P_{33}}\langle FP_{cross}, P_{cross}\rangle.
\end{split}
\end{equation*}
The equality in (\ref{iso3}) follows directly by (\ref{3}).
\end{proof}

\begin{remark}\label{rem4.3}
When $G=G_{2\times 2}^* + e_3\otimes e_3$ then $d=e_3$ and $Dd=De_3 =
0$, so (\ref{iso2}) directly becomes:
\begin{equation}\label{rem}
\mathcal{Q}_2(x', F_{2\times 2}) 
 = \mu |D|^2 + \frac{\lambda\mu}{\lambda+\mu}|\mathrm{tr} D|^2,  
\end{equation}
which is consistent with (\ref{zero0}).
\end{remark}

\begin{remark}
Call $C(x')=G^{-1}(x')F_{2\times 2}^*$ and note that:
\begin{equation*}
\begin{split}
& \mbox{tr}~ D = \mbox{tr} ~C, \quad |D|^2= \mbox{tr}\left(C^2\right)\\
& |Dd|^2 = \langle C^2 G^{-1}e_3, e_3\rangle, \quad |d|^2 = \langle
G^{-1}e_3, e_3\rangle, \quad \langle Dd, d\rangle = \langle C G^{-1}e_3, e_3\rangle
\end{split}
\end{equation*}
Consequently, (\ref{iso2}) can also be equivalently written as:
\begin{equation*}
\mathcal{Q}_2(x', F_{2\times 2}) 
 = \mu\left(\mbox{tr} \left(C^2\right) - 2\frac{\langle C^2 G^{-1}e_3, e_3\rangle}{\langle
G^{-1}e_3, e_3\rangle}  + \frac{\langle C G^{-1}e_3, e_3\rangle^2}{\langle
G^{-1}e_3, e_3\rangle^2}\right) + 
\frac{\lambda\mu}{\lambda+\mu}\left(\mathrm{tr}~ C  - \frac{\langle C G^{-1}e_3, e_3\rangle}{\langle
G^{-1}e_3, e_3\rangle}\right)^2.  
\end{equation*}
\end{remark}
\vspace{\baselineskip}

\begin{center}
{\bf Part B: Other limits}
\end{center}

\section{A characterization of bending: the 3d energy scaling at minimizers}\label{scale}

In this section we deduce the following property, complementary to Corollary \ref{upper}: 
\begin{theorem}\label{lower}
The non-vanishing of the three Riemann curvatures of the metric $G$:
\begin{equation}\label{curv}
\exists x\in\Omega \qquad \Big(|R^3_{112}| + |R^3_{221}| +
|R_{1212}|\Big) (x)\neq 0 
\end{equation}
is equivalent to the lower bound on the energy
scaling at minimizers:
\begin{equation}\label{h2}
\exists c>0 \qquad \inf_{u\in W^{1,2}(\Omega^h,\mathbb{R}^3)} E^h(u)
\geq ch^2.
\end{equation}
\end{theorem}

Recall that the Riemann curvature tensor $R^\cdot_{\cdot\cdot\cdot}$ and its covariant
version $R_{\cdot\cdot\cdot\cdot}$ are given by:
\begin{equation*}
\begin{split}
& R^s_{ijk} = \partial_j\Gamma^s_{ik} - \partial_k\Gamma^s_{ij}
+\sum_{m=1}^3\Gamma^s_{jm}\Gamma^m_{ik} -
\sum_{m=1}^3\Gamma^s_{km}\Gamma^m_{ij}\\
& R_{sijk} = \sum_{m=1}^3G_{sm}R^m_{ijk},
\end{split}
\end{equation*}
while the Christoffel symbols are:
$$\Gamma^i_{kl} = \frac{1}{2} \sum_{m=1}^3G^{im} (\partial_l G_{mk} + \partial_kG_{ml} - \partial_mG_{kl}).$$

\begin{remark} In \cite{lepa} we proved for the metric $G$ having a
2d structure $G=(G_{2\times 2})^*+e_3\otimes e_3$, that condition
(\ref{h2}) is equivalent to the nonimmersability of $G$, i.e. 
nonvanishing of its full Riemann curvature tensor $R$. The reason for
this seemingly more restricitve result is that for such $G$, its
flatness is equivalent to the vanishing of the Gaussian curvature of
$G_{2\times 2}$, i.e. the 2d flatness of the midplate metric
$G_{2\times 2}$. In fact, any isometric immersion of $G$ induces a
flat isometric immersion of $G_{2\times 2}$ whose second fundamental
form $\Pi=0$ trivially satisfies the condition (\ref{zero6}) below.
We see that in the general case the curvatures that
converge to the reduced 2d energy $\mathcal{I}_G$ 
at the scaling $h^2$ are those listed in (\ref{curv}),
rather than all the curvatures which naturally contribute towards
residual 3d energy $E^h$. 
\end{remark}

The proof of Theorem \ref{lower} will follow directly from the next
two theorems, which we present separately for their independent interest.

\begin{theorem}\label{th12}
The following conditions are equivalent:
\begin{itemize}
\item[(i)] The energy functional (\ref{IG}) satisfies:
\begin{equation}\label{zerozero}
\min \mathcal{I}_G=0,
\end{equation}
where the minimum is taken over isometric immersions of
$G_{2\times 2}$ in $\mathbb{R}^3$  of regularity $W^{2,2}$.
\item[(ii)] There exists a $W^{2,2}$ isometric immersion
  $y:\Omega\to\mathbb{R}^3$
of $G_{2\times 2}$ such that:
\begin{equation}\label{zero3}
\mathrm{sym} \Big((\nabla y)^T\nabla \vec b \Big)=0
\qquad \mbox{ a.e. in } \Omega.
\end{equation}
where $\vec b:\Omega\to\mathbb{R}^3$ is uniquely defined by:
$$\det Q>0 \quad \mbox{and} \quad Q^TQ=G, \quad \mbox {where: } ~~Qe_1 = \partial_1y,
~Qe_2 = \partial_2 y, ~ Qe_3 = \vec b.$$
\item[(iii)]  There exists a $W^{2,2}$ isometric immersion of 
$G_{2\times 2}$ in $\mathbb{R}^3$,  whose second fundamental form
$\Pi$ is given by the Christoffel symbols of $G$:
\begin{equation}\label{zero6}
\begin{split}
\Pi_{11} = -\frac{1}{\sqrt{G^{33}}}\Gamma_{11}^3, \qquad
\Pi_{22} = -\frac{1}{\sqrt{G^{33}}}\Gamma_{22}^3, \qquad
\Pi_{12} = -\frac{1}{\sqrt{G^{33}}}\Gamma_{12}^3.
\end{split}
\end{equation}
\end{itemize}
\end{theorem}

\medskip

\begin{corollary}
The unique (up to rigid motions) minimizing immersion in
(\ref{zerozero}) is the immersion $y$ satisfying (\ref{zero3}) and
(\ref{zero6}). This immersion is automatically smooth (up to the boundary of $\Omega$).
\end{corollary}
\begin{proof}
Let $y\in W^{2,2}(\Omega,\mathbb{R}^3)$ satisfy (\ref{cinque}) and
$\mathcal{I}_G(y)=0$.
Denoting by $\vec N\in W^{1,2}(\Omega,\mathbb{R}^3)$ the unit normal to the surface
$y(\Omega)$, we have \cite{HH}:
$$\partial_{ij}y = \sum_{m=1}^2\gamma_{ij}^m \partial_my - \Pi_{ij}\vec N.$$
Since $G_{2\times 2}$ is smooth, then its Christoffel symbols $\gamma_{ij}^m$ are smooth, and
also the coefficients in $\Pi_{ij}$ are smooth according to Theorem \ref{th12}.
Smoothness of $y$ follows then by a bootstrap argument.

Finally, uniqueness of isometric immersion with a prescribed second
fundamental form, completes the proof.
\end{proof}

\begin{theorem}\label{th23}
Conditions in Theorem \ref{th12} are further equivalent to the
vanishing of the following three Riemann curvatures of $G$:
\begin{equation}\label{Riem}
R^3_{112} = R^3_{221} = R_{1212} = 0.
\end{equation}
\end{theorem}

\bigskip 

We now give proofs of Theorem \ref{th12} and Theorem \ref{th23}.

\bigskip

\noindent {\bf Proof of Theorem \ref{th12}.}

{\bf 1.} Condition (i) holds when there exists $y\in W^{2,2}(\Omega, \mathbb{R}^3)$ satisfying
(\ref{cinque}) and:
\begin{equation}\label{zeroo}
\mathcal{I}_G(y)=0.
\end{equation}
The equality (\ref{zeroo}) is clearly equivalent to:
$\mathcal{Q}_2(x', (\nabla y)^T\nabla \vec b) = 0$ holding for a.e. $x'\in \Omega.$
Since $\mathcal{Q}_3(F)=0$ iff $F\in\mathbb{R}^{3\times 3}$ is skew-symmetric, it
follows that (\ref{zeroo}) is further equivalent to (\ref{zero3}), hence
proving (ii).

\medskip

{\bf 2.}  Recall that the matrix field $Q$ in Theorem \ref{thm2}
(ii), whose columns are given by vectors $\partial_1y, \partial_2y$
and $\vec b$, satisfies $Q^TQ=G$. Hence, (\ref{zeroo}) becomes:
\begin{equation}\label{zero1}
\begin{split}
0 &= \langle \partial_1y,\partial_1\vec b\rangle
= \partial_1\langle\partial_1 y, \vec b\rangle -\langle \partial_{11}y,
\vec b\rangle = \partial_1 G_{13} - \langle \partial_{11}y,
\vec b\rangle\\
0 & = \langle \partial_2y,\partial_2\vec b\rangle
= \partial_2\langle\partial_2 y, \vec b\rangle -\langle \partial_{22}y,
\vec b\rangle = \partial_2 G_{23} - \langle \partial_{22}y,
\vec b\rangle\\
0 & = \langle \partial_1y,\partial_2\vec b\rangle + \langle\partial_2y,\partial_1\vec b\rangle
= \partial_1 G_{23} +\partial_2 G_{13} - 2\langle \partial_{12}y,
\vec b\rangle.
\end{split}
\end{equation}

Given $F\in\mathbb{R}^{3\times 3}$, by
$F_{tan}\in\mathbb{R}^{2\times 2}$ let us denote the principal $2\times 2$
minor of $F$, and we let $F_{cross} = (F_{13}, F_{23})^T \in\mathbb{R}^2$.  Call $P=G^{-1}$ Then:
\begin{equation}\label{4}
\begin{split}
& G_{tan}P_{tan} + G_{cross}\otimes P_{cross} = \mbox{Id}_2,\qquad 
G_{tan}P_{cross} + P_{33}G_{cross}=0,
\end{split}
\end{equation}
and so consequently:
\begin{equation}\label{44}
(G_{tan})^{-1} = P_{tan} + (G_{tan})^{-1} G_{cross}\otimes P_{cross}
  = P_{tan} - \frac{1}{P_{33}}P_{cross}\otimes P_{cross}.
\end{equation}
We therefore obtain:
\begin{equation*}
G^{33} = P_{33} = \frac{\det G_{2\times 2}}{\det G},\qquad 
(G_{tan})^{-1} G_{cross} = - \frac{1}{P_{33}} P_{cross}=
-\frac{1}{G^{33}} (G^{13}, G^{23})^T,
\end{equation*}
where as the standard notation is used: $ [G^{ij}]_{i,j:1..3} = G^{-1}= P$.
It follows by (\ref{b}) that:
$$\vec b = -\frac{1}{G^{33}} \left(G^{13}\partial_1y +
  G^{23}\partial_2y\right) + \frac{1}{\sqrt{G^{33}}}\vec N,$$
and hence (\ref{zero1}) becomes:
\begin{equation}\label{zero2}
\begin{split}
\partial_1 G_{13} =
& -\frac{1}{G^{33}}\Big(G^{13}\langle \partial_{11}y, \partial_1y\rangle
  + G^{23}\langle \partial_{11}y, \partial_2y\rangle \Big) - 
  \frac{1}{\sqrt{G^{33}}} \Pi_{11} \\
\partial_2 G_{23} =
& -\frac{1}{G^{33}}\Big(G^{13}\langle \partial_{22}y, \partial_1y\rangle
  + G^{23}\langle \partial_{22}y, \partial_2y\rangle \Big) - 
  \frac{1}{\sqrt{G^{33}}} \Pi_{22} \\
\frac{1}{2}\big(\partial_1 G_{23} +\partial_2 G_{13} \big) =
& -\frac{1}{G^{33}}\Big(G^{13}\langle \partial_{12}y, \partial_1y\rangle
  + G^{23}\langle \partial_{12}y, \partial_2y\rangle \Big) - 
  \frac{1}{\sqrt{G^{33}}} \Pi_{12},
\end{split}
\end{equation}
where we used the fact that the coefficients of the second fundamental
form $\Pi$ of the surface $y(\Omega)$ satisfy: $\Pi_{ij}
= \langle \partial_iy,\partial_j\vec N\rangle =
-\langle\partial_{ij}y,\vec N\rangle$ for $i,j:1.. 2.$

Also, note that $\partial_i G = 2\mbox{sym}\big((\partial_i Q)^T
Q\big)$ for $i=1,2$, from where we deduce:
\begin{equation}\label{zero4}
\begin{split}
&\langle\partial_{11}y, \partial_1 y\rangle =
\frac{1}{2}\partial_1G_{11}, \qquad \langle\partial_{22}y, \partial_2 y\rangle =
\frac{1}{2}\partial_2G_{22}, \\
&\langle\partial_{12}y, \partial_1 y\rangle =
\frac{1}{2}\partial_2G_{11}, \qquad \langle\partial_{12}y, \partial_2 y\rangle =
\frac{1}{2}\partial_1G_{22}, \\
\langle\partial_{11}y, \partial_2 &y\rangle =
\partial_1 G_{12} - \frac{1}{2}\partial_2G_{11}, \qquad \langle\partial_{22}y, \partial_1 y\rangle =
\partial_2 G_{12} - \frac{1}{2}\partial_1G_{22}.
\end{split}
\end{equation}

\medskip

{\bf 3.} We now want to rewrite the equations in (\ref{zero2}) and the formulas
(\ref{zero4}) using the Christoffel symbols $\Gamma_{ij}^m,$ $i,j,m=1..3$ of
the metric $G$.
Recall that, since the Levi-Civita connection is
metric-compatible, we have:
\begin{equation}\label{chris}
\partial_i G_{jk} = \sum_{m=1}^3G_{mk}\Gamma_{ij}^m +
\sum_{m=1}^3G_{mj}\Gamma_{ik}^m,
\end{equation}
\begin{equation}\label{chris2}
\partial_i G^{jk} = -\sum_{m=1}^3G^{mk}\Gamma_{mi}^j -
\sum_{m=1}^3G^{mj}\Gamma_{mi}^k.
\end{equation}
Since $\partial_3 G=0$, it follows that:
$$ \sum_{m=1}^3G_{m1}\Gamma^m_{13} =0,\qquad \sum_{m=1}^3G_{m2}\Gamma^m_{23} =0\qquad
\sum_{m=1}^3(G_{m2}\Gamma^m_{13} + G_{m1}\Gamma^m_{23}) =0. $$
Therefore and  in view of (\ref{chris}), (\ref{zero4}) become:
\begin{equation}\label{zero5}
\begin{split}
&\langle\partial_{11}y, \partial_1 y\rangle =
\sum_{m=1}^3G_{m1}\Gamma_{11}^m, \qquad \langle\partial_{22}y, \partial_2 y\rangle =
\sum_{m=1}^3G_{m2}\Gamma^m_{22},\\
&\langle\partial_{12}y, \partial_1 y\rangle =
\sum_{m=1}^3G_{m1}\Gamma^m_{12},\qquad \langle\partial_{12}y, \partial_2 y\rangle =
\sum_{m=1}^3G_{m2}\Gamma_{12}^m,\\
&\langle\partial_{11}y, \partial_2 y\rangle =
\sum_{m=1}^3G_{m2}\Gamma_{11}^m, \qquad \langle\partial_{22}y, \partial_1 y\rangle =
\sum_{m=1}^3G_{m1}\Gamma_{22}^m.
\end{split}
\end{equation}
By (\ref{zero2}), (\ref{zero5}) and (\ref{chris}) we now obtain:
\begin{equation*}
\begin{split}
\Pi_{11} & = -\frac{1}{\sqrt{G^{33}}}\sum_{m=1}^3 (G^{13}G_{m1} +
G^{23}G_{m2}) \Gamma^m_{11} - \sqrt{G^{33}} \sum_{m=1}^3 G_{m3}\Gamma_{11}^m,\\
\Pi_{22} & = -\frac{1}{\sqrt{G^{33}}}\sum_{m=1}^3 (G^{13}G_{m1} +
G^{23}G_{m2}) \Gamma^m_{22} - \sqrt{G^{33}} \sum_{m=1}^3 G_{m3}\Gamma_{22}^m,\\
\Pi_{12} & = -\frac{1}{\sqrt{G^{33}}}\sum_{m=1}^3 (G^{13}G_{m1} +
G^{23}G_{m2}) \Gamma^m_{12} - \sqrt{G^{33}} \sum_{m=1}^3 G_{m3}\Gamma_{12}^m.
\end{split}
\end{equation*}
Since $\sum_{i=1}^3 G^{i3}G_{mi} = \delta_{m3}$, we conclude (\ref{zero6})
and note that it is equivalent to (\ref{zeroo}).
\endproof

\bigskip

\noindent {\bf Proof of Theorem \ref{th23}.}

Clearly, Theorem \ref{th12} (iii) is
equivalent (see \cite{HH} for details) to the satisfaction of the Codazzi-Mainardi equations
for the 2d metric $G_{2\times 2}$ and the second fundamental form $\Pi$:
\begin{equation}\label{CM}
\begin{split}
&\partial_2\left(\frac{1}{\sqrt{G^{33}}} \Gamma_{11}^3\right)
- \partial_1\left(\frac{1}{\sqrt{G^{33}}} \Gamma_{12}^3\right) 
= \frac{1}{\sqrt{G^{33}}} \left( \sum_{m=1}^2\Gamma^3_{1m}\gamma^m_{12} -
\sum_{m=1}^2\Gamma^3_{2m}\gamma^m_{11} \right)\\
&\partial_2\left(\frac{1}{\sqrt{G^{33}}} \Gamma_{12}^3\right)
- \partial_1\left(\frac{1}{\sqrt{G^{33}}} \Gamma_{22}^3\right) 
= \frac{1}{\sqrt{G^{33}}} \left( \sum_{m=1}^2\Gamma^3_{1m}\gamma^m_{22} -
\sum_{m=1}^2\Gamma^3_{2m}\gamma^m_{12} \right)
\end{split}
\end{equation}
together with the Gauss equation:
\begin{equation}\label{Gauss}
\Gamma_{11}^3\Gamma_{22}^3 - (\Gamma_{12}^3)^2 = G^{33} \kappa \det
G_{2\times 2}.
\end{equation}
Above, by $\gamma_{ij}^k$ we denote the Christoffel symbols of the
metric $G_{2\times 2}$, while $\kappa=\kappa(G_{2\times 2})$ is the
Gaussian curvature of $G_{2\times 2}$. 
We now prove that (\ref{CM}) and
(\ref{Gauss}) are equivalent to (\ref{Riem}).

\medskip

{\bf 1.} We first relate the Christoffel symbols $\gamma_{kl}^s$ with
$\Gamma_{kl}^s$. By (\ref{4}), the inverse matrix
$[g^{ij}]_{i,j=1..2}=[(G_{2\times 2})^{-1}]_{ij}$ is given by:
$g^{ij} = G^{ij} - \frac{1}{G^{33}} G^{3i}G^{3j}.$
Hence:
\begin{equation}\label{cs}
\begin{split}
\gamma^s_{kl} & = \frac{1}{2}\sum_{m=1}^2 g^{sm}(\partial_lG_{mk}
+ \partial_kG_{ml} - \partial_mG_{kl}) \\ & = \frac{1}{2} \sum_{m=1}^2 G^{sm}(\partial_lG_{mk}
+ \partial_kG_{ml} - \partial_mG_{kl}) - \frac{1}{2} \frac{G^{3s}}{G^{33}} \sum_{m=1}^2 G^{3m}\big(\partial_lG_{mk}
+ \partial_kG_{ml} - \partial_mG_{kl}\big) \\ & = \Gamma^s_{kl} -
G^{3s}(\partial_lG_{3k} +\partial_kG_{3l}) -
\frac{G^{3s}}{G^{33}}\Big(\Gamma_{kl}^3 - G^{33}(\partial_lG_{3k}
  +\partial_kG_{3l})\Big) \\ & = \Gamma^s_{kl} - \frac{G^{3s}}{G^{33}} \Gamma_{kl}^3.
\end{split}
\end{equation}
Also, note that by (\ref{chris2}), for $i=1,2$ we have: 
\begin{equation}\label{zero66}
\sqrt{G^{33}}\partial_i (\frac{1}{\sqrt{G^{33}}}) =
-\frac{1}{2}\frac{\partial_i G^{33}}{G^{33}} = \frac{1}{G^{33}}
\sum_{m=1}^3G^{m3}\Gamma^3_{mi}.
\end{equation}

\medskip

{\bf 2.}  The first equation in (\ref{CM}) now becomes:
\begin{equation*}
\begin{split}
\partial_2\Gamma_{11}^3 - \partial_1 &\Gamma_{12}^3 - \frac{1}{2}
\left(\frac{\partial_2G^{33}}{G^{33}} \Gamma_{11}^3
- \frac{\partial_1G^{33}}{G^{33}} \Gamma_{12}^3\right) \\ & 
\qquad\qquad = \left( \sum_{m=1}^2\Gamma^3_{1m}\Gamma^m_{12} -
\sum_{m=1}^2\Gamma^3_{2m}\Gamma^m_{11} \right) +
\frac{G^{32}}{G^{33}}(\Gamma^3_{11}\Gamma^3_{22} - (\Gamma_{12}^3)^2)
\end{split}
\end{equation*}
Therefore, in view of (\ref{zero66}) we obtain:
\begin{equation*}
\begin{split}
R^3_{121} & = 
\partial_2\Gamma_{11}^3 - \partial_1 \Gamma_{12}^3 +
\sum_{m=1}^3 (\Gamma^3_{2m}\Gamma^m_{11} - \Gamma^3_{1m}\Gamma^m_{12})
\\ & = \frac{1}{G^{33}} \left(G^{33}(\Gamma^3_{23}\Gamma^3_{11}
  - \Gamma^3_{13}\Gamma^3_{12}) + \frac{1}{2}
(\partial_2G^{33}\Gamma_{11}^3 - \partial_1G^{33}\Gamma_{12}^3) 
+ G^{32}(\Gamma^3_{11}\Gamma^3_{22} - (\Gamma_{12}^3)^2)\right) \\ & =
\frac{1}{G^{33}} \left(\sum_{m=1}^2(G^{m3}\Gamma^3_{m1}\Gamma^3_{12} -
  G^{m3}\Gamma^3_{m2}\Gamma^3_{11}) 
+ G^{32}(\Gamma^3_{11}\Gamma^3_{22} - (\Gamma_{12}^3)^2)\right),
\end{split}
\end{equation*}
which gives $R^3_{121}=0$ by direct inspection.
Similarly, the second equation in (\ref{CM}) yields:
\begin{equation*}
\begin{split}
\partial_2\Gamma_{12}^3 - \partial_1 &\Gamma_{22}^3 - \frac{1}{2}
\left(\frac{\partial_2G^{33}}{G^{33}} \Gamma_{12}^3
- \frac{\partial_1G^{33}}{G^{33}} \Gamma_{22}^3\right) \\ & 
\qquad\qquad = \left( \sum_{m=1}^2\Gamma^3_{1m}\Gamma^m_{22} -
\sum_{m=1}^2\Gamma^3_{2m}\Gamma^m_{12} \right) +
\frac{G^{31}}{G^{33}}(\Gamma^3_{11}\Gamma^3_{22} - (\Gamma_{12}^3)^2)
\end{split}
\end{equation*}
Consequently, using (\ref{zero66}) as before:
\begin{equation*}
\begin{split}
R^3_{221} & = 
\partial_2\Gamma_{12}^3 - \partial_1 \Gamma_{22}^3 +
\sum_{m=1}^3 (\Gamma^3_{2m}\Gamma^m_{12} - \Gamma^3_{1m}\Gamma^m_{22})
\\ & = \frac{1}{G^{33}} \left(G^{33}(\Gamma^3_{23}\Gamma^3_{12}
  - \Gamma^3_{13}\Gamma^3_{22}) + \frac{1}{2}
(\partial_2G^{33}\Gamma_{12}^3 - \partial_1G^{33}\Gamma_{22}^3) 
- G^{31}(\Gamma^3_{11}\Gamma^3_{22} - (\Gamma_{12}^3)^2)\right) \\ & =
\frac{1}{G^{33}} \left(\sum_{m=1}^2(G^{m3}\Gamma^3_{m1}\Gamma^3_{22} -
  G^{m3}\Gamma^3_{m2}\Gamma^3_{12}) 
- G^{31}(\Gamma^3_{11}\Gamma^3_{22} - (\Gamma_{12}^3)^2)\right),
\end{split}
\end{equation*}
which implies $R^3_{221}=0$.

\medskip

{\bf 3.} We now turn to proving equivalence of (\ref{Gauss}) with
$R_{1212}=0$. Denoting $r^s_{ijk}$ and $r_{sijk}$ the Riemann
curvatures of the metric $G_{2\times 2}$ (where $i,j,k,s=1..2$) we
obtain:
$$\kappa\det G_{2\times 2} = r_{1212}= G_{11}r^1_{212} + G_{12}r^2_{212}$$
Further, for $i=1,2$ we get by (\ref{cs}) and (\ref{chris2}):
\begin{equation*}
\begin{split}
r^i_{212}  = & ~\partial_1(\Gamma^i_{22} - \frac{G^{3i}}{G^{33}}\Gamma_{22}^3) -
\partial_2 (\Gamma^i_{12} - \frac{G^{3i}}{G^{33}}\Gamma_{12}^3) \\ &
+\sum_{m=1}^2(\Gamma^i_{1m} - \frac{G^{3i}}{G^{33}}\Gamma_{1m}^3)
(\Gamma^m_{22} - \frac{G^{3m}}{G^{33}}\Gamma_{22}^3) 
- \sum_{m=1}^2(\Gamma^i_{2m} - \frac{G^{3i}}{G^{33}}\Gamma_{2m}^3)
(\Gamma^m_{12} - \frac{G^{3m}}{G^{33}}\Gamma_{12}^3)\\
= & ~R^i_{212} - \frac{G^{3i}}{G^{33}}(\partial_1\Gamma^3_{22} - \partial_2 \Gamma^3_{12})  
+\left(\frac{G^{1i}}{G^{33}} - \frac{G^{3i} G^{31}}{(G^{33})^2}\right)
(\Gamma^3_{11} \Gamma_{22}^3 - (\Gamma^3_{12})^2) \\ & - \frac{G^{3i}}{G^{33}} 
\sum_{m=1}^3 (\Gamma^3_{1m}\Gamma^m_{22} - \Gamma^3_{2m}\Gamma^m_{12}).
\end{split}
\end{equation*}
Consequently, the Gauss equation (\ref{Gauss}) yields:
\begin{equation*}
\begin{split}
R_{1212}  & = G_{11}R^1_{212} + G_{12}R^2_{212} + G_{13} R^3_{212}\\ & = \kappa\det
G_{2\times 2} + G_{11}(R^1_{212} - r^1_{212}) + G_{12}(R^2_{212} -
r^2_{212}) + G_{13} R^3_{212}\\ & = G_{13} R^3_{212} + \frac{1}{G^{33}} (\Gamma^3_{11} \Gamma_{22}^3 -
(\Gamma^3_{12})^2) - G_{13} (\partial_1\Gamma^3_{22} - \partial_2
\Gamma^3_{12})  \\ & \quad - \left( \frac{(1-G^{13}G_{13})}{G^{33}} + \frac{G^{33}G_{13} G^{31}}{(G^{33})^2}\right)
(\Gamma^3_{11} \Gamma_{22}^3 - (\Gamma^3_{12})^2)
 - G_{31} \sum_{m=1}^3 (\Gamma^3_{1m}\Gamma^m_{22} -
 \Gamma^3_{2m}\Gamma^m_{12}) \\ &
= G_{13} R^3_{212} - G_{13} \left(\partial_1\Gamma^3_{22} - \partial_2
\Gamma^3_{12} + \sum_{m=1}^3 (\Gamma^3_{1m}\Gamma^m_{22} -
 \Gamma^3_{2m}\Gamma^m_{12}) \right) = 0.
\end{split}
\end{equation*}
Note that we did not use the fact that $R^3_{212}=0$ in the above
calculation. This completes the proof of Theorem \ref{th23}.
\endproof

\section{Examples}\label{examples}

\begin{example}\label{ex1} 
Let $\lambda:\bar\Omega\to\mathbb{R}$ be a smooth positive function and define:
\begin{equation}\label{e1}
G(x', x_3)= G(x') = \mbox{diag} \big(1, 1, \lambda(x')\big).
\end{equation}
Clearly, the 2d metric $G_{2\times 2}=\mbox{Id}_2$ has an isometric
immersion $y_0(x')=x'$ with the second fundamental form $\Pi = 0$.
On the other hand:
$$\forall i,j:1..2\qquad \Gamma^3_{ij}= \frac{1}{2\lambda} (\partial_i
G_{3j} + \partial_j G_{3i}) = 0,$$
and we see that both conditions (i) and (ii) in  Theorem \ref{th12}
are satisfied,  so that $\mathcal{I}_G(y_0) =0$.

We can further check directly that the only possibly non-zero Christoffel symbols are:
$$\forall i:1..2\qquad \Gamma^i_{33}= -\frac{1}{2}\partial_i\lambda,
\qquad \Gamma^3_{i3} = \frac{1}{2\lambda} \partial_i\lambda.$$
In particular, it easily follows that: $R^3_{121} = R^3_{221} = R_{1212} =
0$, which is consistent with Theorem \ref{th23}.
At the same time $G$ is, in general,  non-immersable. To see this,
recall that the Ricci curvatures are given by:
$R_{ij} = \sum_{l=1}^3(\partial_l\Gamma_{ij}^l
- \partial_j\Gamma_{il}^l) +  \sum_{l, m=1}^3
(\Gamma_{ij}^l\Gamma_{lm}^m - \Gamma_{il}^m\Gamma_{jm}^l)$, for
$i,j=1\ldots 3$. In the present case, we have:
\begin{equation*}
\begin{split}
& R_{11} = \frac{1}{4\lambda^2} \big((\partial_1\lambda)^2 -
2\lambda(\partial_{11}\lambda)\big), \quad R_{12}=
\frac{1}{4\lambda^2} \big((\partial_1\lambda)(\partial_2\lambda) - 
2\lambda(\partial_{12}\lambda)\big), \quad  R_{13} = R_{23} = 0,\\
& R_{22} = \frac{1}{4\lambda^2} \big((\partial_2\lambda)^2 -
2\lambda(\partial_{22}\lambda)\big), \quad 
R_{33} = - \frac{1}{4\lambda} \big((\partial_1\lambda)^2 -
2\lambda(\partial_{11}\lambda)  + (\partial_2\lambda)^2 -
2\lambda(\partial_{22}\lambda)\big).
\end{split}
\end{equation*}
We hence see that $G$ is immersable if and only if:
\begin{equation}\label{pies}
M:=\nabla^2\lambda -
\frac{1}{2\lambda}\nabla\lambda\otimes\nabla\lambda \equiv
0\quad\mbox{ in } \Omega.
\end{equation}

Let us now consider the scaling of the $3$d non-Euclidean
energy studied in this paper:
$$E^h(u^h) = \frac{1}{h}\int_{\Omega^h} W(\nabla
u^h\sqrt{G}^{-1})~\mbox{d}x,\qquad \sqrt{G}^{-1}=\mbox{diag}(1,1,\frac{1}{\sqrt{\lambda}}),$$
at the following sequence of smooth deformations of $\Omega^h$:
\begin{equation}\label{koko}
u^h(x', x_3) = x' + \big(- \frac{x_3^2}{4}\partial_1\lambda, 
- \frac{x_3^2}{4}\partial_2\lambda, \sqrt{\lambda}~x_3\big)^T.
\end{equation}
We have: $\big((\nabla u^h)\sqrt{G}^{-1}\big)_{2\times 2} = \mbox{Id}_2 -
\frac{x_3^2}{4} \nabla^2\lambda$, and: 
$$\big((\nabla u^h)\sqrt{G}^{-1}\big) e_3 = \big(- \frac{x_3}{2}\frac{\partial_1\lambda}{\sqrt{\lambda}}, 
- \frac{x_3}{2}\frac{\partial_2\lambda}{\sqrt{\lambda}}, 1\big)^T,
\qquad \big((\nabla u^h)\sqrt{G}^{-1}\big)^T e_3 = \big(\frac{x_3}{2}\frac{\partial_1\lambda}{\sqrt{\lambda}}, 
\frac{x_3}{2}\frac{\partial_2\lambda}{\sqrt{\lambda}}, 1\big)^T.$$
Recall that for every $F=\mbox{Id}_3 + \mathcal{A}\in\mathbb{R}^{3\times 3}$
when $\mathcal{A}$ is sufficiently small, we have: 
$ \mbox{dist}(F, SO(3))=
|\sqrt{F^TF} - \mbox{Id}|  = |\sqrt{\mbox{Id} + 2\mbox{sym }
  \mathcal{A} + \mathcal{A}^T\mathcal{A}} - \mbox{Id}| = 
\big |\mbox{sym } \mathcal{A} + \frac{1}{2}\mathcal{A}^T\mathcal{A} + {o}(|\mbox{sym } \mathcal{A} +
\frac{1}{2}\mathcal{A}^T\mathcal{A}|)\big|$.
Consequently:
$$W\big((\nabla u^h)\sqrt{G}^{-1}\big) \leq C\mbox{dist}^2\big((\nabla
u^h)\sqrt{G}^{-1}, SO(3)\big)\leq C x_3^4, $$
and therefore:
\begin{equation}\label{ex_conclusion}
\inf E^h \leq E^h(u^h) \leq \frac{C}{h} \int_{-h/2}^{h/2}
x_3^4~\mbox{d}x_3 = Ch^4,
\end{equation}
for any choice of $\lambda$ in (\ref{e1}). In section \ref{exactscaling}
we will show that this scaling is also optimal, provided that the condition
(\ref{pies}) does not hold.
\end{example}

\medskip

\begin{example}\label{ex2}
Let $\lambda:\bar\Omega\to\mathbb{R}$ be a smooth positive function
and consider the metric:
\begin{equation}\label{e2}
G(x', x_3)= G(x') =\lambda(x') \mbox{Id}_3. 
\end{equation}
One checks directly that $\Gamma^i_{kl} = \frac{1}{2\lambda}
(\delta_{ik}\partial_l\lambda + \delta_{il}\partial_k\lambda -
\delta_{kl}\partial_i\lambda) = \delta_{ik}\partial_l f +
\delta_{il}\partial_k f - \delta_{kl}\partial_i f$, where we denote
$f=\frac{1}{2}\log\lambda$. We directly compute:
$$R^3_{112}= R^3_{221} = 0\quad \mbox{ and } \quad R_{1212} =
-\frac{1}{2}\lambda\Delta(\log \lambda) = \lambda^2\kappa(\lambda\mbox{Id}_2).$$
Therefore, condition (\ref{Riem}) which is equivalent to $\min
\mathcal{I}_G=0$ according to Theorem \ref{th23}, holds if and only if:
\begin{equation}\label{e22}
\Delta(\log\lambda) = 0,
\end{equation}
or equivalently, when the 2d metric $G_{2\times 2}=\lambda\mbox{Id}_2$
is flat (immersable in $\mathbb{R}^2$). Note also that since
$\Gamma^3_{ij}=0$ for $i,j:1..2$ then this is precisely the case when
(\ref{zero6}) of Theorem \ref{th12} is satisfied.

We now compute the Ricci curvature of $G$ using the conformal rescaling formula:
\begin{equation*}
\begin{split}
\mbox{Ric}(G) & = \mbox{Ric}(e^{2f}\mbox{Id}_3)= -(\nabla^2f -\nabla
f\otimes\nabla f) - \big(\Delta f - |\nabla f|^2\big)\mbox{Id}_3 \\ & =
\Big(-2(\Delta f)\mbox{Id}_2 + \mbox{cof}(\nabla^2f - \nabla f\otimes
\nabla f)\Big)^* - (\Delta f + |\nabla f|^2) ~e_3\otimes e_3.
\end{split}
\end{equation*}
We observe that $G$ is immersable iff $\mbox{Ric}(G)=0$,
i.e. when $\nabla f = 0$, which is equivalent to:
\begin{equation}\label{e23}
\lambda \equiv \mbox{const}.
\end{equation}
Clearly (\ref{e23}) implies (\ref{e22}), but conversely: there exist nonimmersable
metrics $G$ for which (\ref{e22}) holds i.e. for which the minimum of
the residual energy $\mathcal{I}_G$ is $0$, and it is attained by the
unique (up to rigid motions) smooth isometric immersion
$y:\Omega\to\mathbb{R}^2$  of $\lambda\mbox{Id}_2$ .

As in Example \ref{ex1}, we now consider scaling of the 3d energies
$E^h$, assuming (\ref{e22}). Define:
\begin{equation}\label{ciag}
u^h(x', x_3) = y(x')^* + x_3\sqrt{\lambda}e_3 -\frac{x_3^2}{4}
\Big((\nabla y)^{-1, T}\nabla\lambda\Big)^*.
\end{equation}
We easily compute that $\big(\nabla u^h(x', x_3)\big)_{2\times 2} = \nabla y +
\mathcal{O}(x_3^2)$ and:
\begin{equation*}
\begin{split}
\big(\nabla u^h(x', x_3)\big)e_3 & = 
\Big(- \frac{x_3}{2} \langle (\nabla y)^{-1, T}\nabla \lambda,
e_1\rangle, - \frac{x_3}{2} \langle (\nabla y)^{-1, T}\nabla \lambda, e_2\rangle, 
\sqrt{\lambda}\Big)^T, \\
\big(\nabla u^h(x', x_3)\big)^T e_3 & = 
\Big(\frac{x_3}{2\sqrt{\lambda}} \partial_1 \lambda, \frac{x_3}{2\sqrt{\lambda}} \partial_2 \lambda,
\sqrt{\lambda}\Big)^T.
\end{split}
\end{equation*}
Since $(\nabla y)^T\nabla y = \lambda\mbox{Id}_2$ it follows that:
\begin{equation*}
\begin{split}
\Big((\nabla u^h)^T \nabla u^h\Big)_{2\times 2} & = \lambda \mbox{Id}_2 +\mathcal{O}(x_3^2)\\
\forall i=1..2\quad \Big((\nabla u^h)^T \nabla u^h\Big)_{3i} & = -\frac{x_3}{2}
\langle \partial_i y, (\nabla y)^{-1, T}\nabla \lambda\rangle +
\frac{x_3}{2} \partial_i\lambda + \mathcal{O}(x_3^2) = \mathcal{O}(x_3^2) \\
\Big((\nabla u^h)^T \nabla u^h\Big)_{33} & = \lambda + \mathcal{O}(x_3^2).
\end{split}
\end{equation*}
Therefore it follows,  by polar decomposition theorem:
\begin{equation*}
\begin{split}
W\big((\nabla u^h)\sqrt{G}^{-1}\big) & \leq C\mbox{dist}^2\Big(\frac{1}{\sqrt{\lambda}}\nabla
u^h, SO(3)\Big)\leq C \mbox{dist}^2\Big(\sqrt{\frac{1}{\lambda}(\nabla
u^h)^T\nabla u^h}, SO(3)\Big) \\ & 
=  C \mbox{dist}^2\Big(\sqrt{\mbox{Id}_3 + \mathcal{O}(x_3^2)},
SO(3)\Big) \leq C x_3^4,
\end{split}
\end{equation*}
which again yields the scaling $h^4$, precisely as in (\ref{ex_conclusion}).
\end{example}

\begin{remark}
A more general example of $G$ in the same spirit as above,
is: $G(x')=G_{2\times2}^*+\lambda(x')e_3\otimes e_3$ with
$G_{2\times 2}$  immersable in $\mathbb{R}^2$. Since $\Gamma^3_{ij}=0$
for $i,j=1..2$, we see that $\min\mathcal{I}_G = 0$, in virtue of Theorem \ref{th12}.
On the other hand, one can check directly that taking smooth $y:\Omega\to\mathbb{R}^2$ such that $(\nabla
y)^T \nabla y = G_{2\times2}$ and defining the $3$d deformations $u^h$
as in (\ref{ciag}), it again follows: $W\big((\nabla
u^h)\sqrt{G}^{-1}\big)  \leq  C x_3^4$. Consequently, the
same energy scaling as in (\ref{ex_conclusion}) is valid here as well.
\end{remark}

\medskip

\begin{example}\label{ex21} 
{\bf (i).}  Let $\lambda_1, \lambda_2, \lambda_3:\bar\Omega\to\mathbb{R}$ be 
smooth functions such that $\lambda_3 > \lambda_1^2 + \lambda_2^2$. Define:
\begin{equation}\label{e21}
G(x', x_3)= G(x') = \left[\begin{array}{ccc} 1 & 0 & \lambda_1\\ 0 & 1
    & \lambda_2 \\ \lambda_1 & \lambda_2 & \lambda_3 \end{array}\right].
\end{equation}
We directly compute the Christoffel symbols involved in (\ref{zero6}):
\begin{equation}\label{pom}
\Gamma_{11}^3 = \frac{\partial_1\lambda_1}{\lambda_3 - (\lambda_1^2
  + \lambda_2^2)}, \qquad \Gamma_{12}^3 =
\frac{\frac{1}{2}(\partial_1\lambda_2 + \partial_2\lambda_1)}{\lambda_3 - (\lambda_1^2
  + \lambda_2^2)}, \qquad \Gamma_{22}^3 =
\frac{\partial_2\lambda_2}{\lambda_3 - (\lambda_1^2 + \lambda_2^2)}. 
\end{equation}
Hence, if $|\partial_1\lambda_1| + |\partial_2\lambda_2|\not\equiv 0$
in $\Omega$,  it follows by Theorem \ref{th12} (iii), that the isometric
immersion $y_0(x')=x'$ of $G_{2\times 2}=\mbox{Id}_2$ is certainly not
the immersion for which $\mathcal{I}_G(y_0) = 0$. Of course, this does not
preclude the possibility that there exists another immersion
$y:\Omega\to \mathbb{R}^3$ of $G_{2\times 2}$ (now necessarily non-flat), for
which $\mathcal{I}_G(y)=0$.  As we shall see below, both scenarios are possible.

\medskip

{\bf (ii).}  Consider a subcase of (\ref{e21}), where $\lambda_1=0$
and $\lambda_3 = \lambda_2^2 +1$, so that:
\begin{equation*}
G(x', x_3)= G(x') = \left[\begin{array}{ccc} 1 & 0 & 0\\ 0 & 1
    & \lambda_2 \\ 0 & \lambda_2 & \lambda_2^2 + 1 \end{array}\right].
\end{equation*}
Consequently, by (\ref{pom}):
\begin{equation}\label{pom2} 
\Gamma_{11}^3 = 0, \qquad \Gamma_{12}^3 =
\frac{1}{2} \partial_1\lambda_2, \qquad \Gamma_{22}^3 = \partial_2\lambda_2. 
\end{equation}
We further compute:
$$ \Gamma_{11}^1 = \Gamma_{11}^2 = 0, 
\qquad \Gamma_{13}^3 = \frac{1}{2}\lambda_2\partial_1\lambda_2, 
\qquad \Gamma_{12}^2 = -\frac{1}{2}\lambda_2\partial_1\lambda_2, $$
and:
$$R_{112}^3 = \partial_1\Gamma_{12}^3 + (\Gamma_{12}^3\Gamma_{12}^2 +
\Gamma_{13}^3\Gamma_{12}^3 ) - (\Gamma_{12}^3\Gamma_{11}^1 +
\Gamma_{22}^3\Gamma_{11}^2 ) = \frac{1}{2}\partial_{11}\lambda_2.$$
We see that when $\partial_{11}\lambda_2\not\equiv 0$ in $\Omega$, then
$\min\mathcal{I}_G>0$, in view of Theorem \ref{th23}. In particular,
there is no isometric immersion of $G_{2\times 2}$ satisfying (\ref{zero6}).

\medskip

{\bf (iii).}  Consider now a further subcase of (\ref{e21}), with
$\lambda_1 =0$, $\lambda_2 = -x_2$ and $\lambda_3 = x_2^2 + 1$, so that:
\begin{equation*}
G(x', x_3)= G(x_2) = \left[\begin{array}{ccc} 1 & 0 & 0\\ 0 & 1
    & - x_2 \\ 0 & -x_2 & x_2^2 + 1 \end{array}\right].
\end{equation*}
By (\ref{pom2}) we get: $\Gamma_{11}^3 = \Gamma_{12}^3 = 0$
and  $\Gamma_{22}^3 = -1$.
Let $y(x_1, x_2) = (x_1, \sin x_2, \cos x_2)$ be an isometric
immersion of $G_{2\times 2}$ into a cylinder. This immersion has the second
fundamental form $\Pi$:
$$\Pi = (\nabla y)^T \nabla \vec N = \left[\begin{array}{cc} 0 & 0\\ 0
    & 1 \end{array}\right] =
-\frac{1}{\sqrt{G^{33}}}\left[\begin{array}{cc} \Gamma^3_{11} &
    \Gamma^3_{12} \\ \Gamma^3_{12} & \Gamma^3_{22} \end{array}\right]. $$ 
Hence, by Theorem \ref{th12} (iii) it follows that $\mathcal{I}_G(y) = 0$.

In particular, $R_{112}^3 = R_{221}^3 = R_{1212}\equiv 0$ in
$\Omega$. Metric $G$ is, however, nonimmersable, as a direct
calculation of its scalar Ricci curvature shows:
$$S = G^{11} R_{11} + G^{22} R_{22} + G^{33}R_{33} + 2G^{23}R_{23} =
-2 + 2x_2^2 \not\equiv 0.$$

\end{example}

\medskip

\begin{example}\label{ex3}
In this example we will have $G_{2\times 2}$ nonimmersable in
$\mathbb{R}^2$. Let $\bar \Omega\subset \{(x_1, x_2)\in\mathbb{R}^2;~x_1>x_2>0\}$ and define $G$:
\begin{equation}\label{e3}
G(x_1, x_2)=\left[\begin{array}{ccc} 1+x_1^2 & x_1x_2 & b_1 + x_1b_3\\
x_1x_2 & 1+ x_2^2 & b_2+x_2b_3\\
b_1+x_1b_3 & b_2 + x_2b_3 & |\vec b|^2 \end{array}\right] \quad
\mbox{where } \vec b = \Big( -\frac{1}{3}x_1^3, \frac{1}{3}x_2^3, \frac{1}{2}(x_1^2-x_2^2)\Big)^T. 
\end{equation}
We see that:
$$y(x_1, x_2) = \big(x_1, x_2, \frac{1}{2}(x_1^2+x_2^2)\big)$$
is an isometric immersion of $G_{2\times 2}$ in $\mathbb{R}^3$.
Therefore: 
$$\kappa(G_{2\times 2}) = \frac{\partial_{11}y_3 \partial_{22}y_3 -
  (\partial_{12}y_3)^2}{(1+|\nabla y_3|^2)^2} = \big(1+x_1^2 +
x_2^2\big)^{-2}\neq 0.$$
By Theorem \ref{th12}, we have: $\min \mathcal{I}_G =
0$ iff (\ref{zero3}) holds. This is equivalent to 
$\mbox{sym}(\nabla \vec b_{tan} + \nabla b_3\otimes (x_1, x_2)) = 0$,
and further to:
$$\mbox{sym}(\nabla b_3\otimes (x_1, x_2)) = -\mbox{sym}\nabla \vec b_{tan}.$$
Given a scalar field $b_3$, there exists $\vec b_{tan}$ so that the above
condition is satisfied iff:
$$0 = \mbox{curl}^T\mbox{curl}(\nabla b_3\otimes (x_1, x_2)) = -\Delta b_3.$$
We see that indeed $b_3$ in (\ref{e3}) is harmonic, and that $(b_1,
b_2)$ satisfiy: $\mbox{sym}\nabla \vec b_{tan} = \mbox{diag}(-x_1^2,
x_2^2)$, which implies $\mathcal{I}_G(y)=0$, for the $3$d metric $G=Q^TQ$, where
$Q=\left[ \partial_1y, \partial_2 y, \vec b\right]$. Note also that $\det
Q>0$ in $\bar\Omega$. Hence, $G$ is
given by (\ref{e3}). One can check that $G$ is nonimmersable in $\mathbb{R}^3$, by
calculating its Ricci curvature (we have used Maple\textregistered). In particular,
the scalar Ricci curvature of $G$ equals:
$$S=\sum_{i,j=1..3} G^{ij}R_{ij} = \frac{12}{2x_1^2 + 2x_2^2 + 3}\neq 0.$$
\end{example}

\medskip

In section \ref{lg} we will discuss other examples of types of metric $G$,
motivated by the modelling of the nematic glass.

\section{Scaling analysis for Example \ref{ex1}}\label{exactscaling}

In this section we will prove the optimality of the scaling $h^4$ in 
(\ref{ex_conclusion}) for every non-immersable metric tensor $G$ of the form 
in Example \ref{ex1}. 

\begin{theorem}\label{optimal}
Let $\lambda:\bar\Omega\to\mathbb{R}$ be a smooth positive function
such that: 
\begin{equation*}\label{ges}
M:=\nabla^2\lambda - \frac{1}{2\lambda} \nabla\lambda\otimes
\nabla\lambda \not\equiv 0 \quad \mbox{ in } \Omega.
\end{equation*}
Define $G(x', x_3) = G(x') = \mathrm{diag} (1,1,\lambda(x'))$. Then
$G$ is non-immersable and
there exist $c,C>0$ such that:
\begin{equation}\label{glupiages}
ch^4\leq \inf E^h \leq Ch^4.
\end{equation}
\end{theorem}
\begin{proof}
The equivalence of nonimmersability of $G$ with the condition
$M\not\equiv 0$, as well as the upper bound has been established in
Example \ref{ex1}, see (\ref{pies}) and (\ref{ex_conclusion}). Hence
it remains to prove the lower bound of (\ref{glupiages}).
Recall that the minimizing isometric immersion of $G_{2\times 2}$ here
is $y(x') = (x',0)$, whereas other quantities in the proofs of Theorems
\ref{thm2} and \ref{thm3} are:
\begin{equation*}
\begin{split}
&\vec b=\sqrt{\lambda}e_3, \qquad Q=A=\mathrm{diag} (1,1,
\sqrt{\lambda})\\
&\vec d=(-\frac{\partial_1\lambda}{2}, 
-\frac{\partial_2\lambda}{2}, 0)^T, \qquad 
B(x') = \left[\begin{array}{ccc} \partial_1\vec b, \partial_2\vec b,
    \vec d\end{array}\right] = \left[\begin{array}{ccc} 0 & 0 & -\frac{\partial_1\lambda}{2}\\
0 & 0 & -\frac{\partial_2\lambda}{2}\\
\frac{\partial_1\lambda}{2\sqrt{\lambda}} &
\frac{\partial_2\lambda}{2\sqrt{\lambda}} & 0\end{array}\right].
\end{split}
\end{equation*}

\medskip

{\bf 1.}  Let $u^h\in W^{1,2}(\Omega^h, \mathbb{R}^3)$ be any sequence of
deformations, satisfying:
\begin{equation}\label{ges2}
E^h(u^h)\leq Ch^4.
\end{equation}
Consider a smooth diffeomorphism $\phi:\Omega^{h_0}\to\mathbb{R}^3$ (for
$h_0>0$ sufficiently small) given as in (\ref{koko}):
$$ \phi(x_1, x_2, x_3) = (x_1 - \frac{x_3^2}{4}\partial_1\lambda,~ x_2 - \frac{x_3^2}{4}\partial_2\lambda, 
~\sqrt{\lambda}x_3)^T.$$
We will write $\mathcal{U}^h = \phi(\Omega^h)$. Since $\nabla \phi = Q+x_3 B
- \frac{x_3^2}{4} (\nabla^2\lambda)^*$, it follows that:
$$(\nabla \phi) A^{-1} = \mbox{Id}_3 + x_3 BA^{-1} - \frac{x_3^2}{4} (\nabla^2\lambda)^*.$$
Further, noting that $BA^{-1}\in so(3)$, by polar decomposition
theorem we get:
$$\exists R(x)\in SO(3)\qquad 
(\nabla \phi) A^{-1} = \sqrt{(\nabla \phi) A^{-1} ((\nabla \phi) A^{-1})^T }
R(x) 
= (\mbox{Id}_3 + \mathcal{O}(x_3^2)) R(x). $$
Then:
\begin{equation}\label{gesa}
\begin{split}
E^h(u^h) & \geq c\frac{1}{h}\int_{\Omega^h} \mbox{dist}^2\left((\nabla
  u^h) A^{-1}, SO(3)\right)~\mbox{d}x \\
& = c\frac{1}{h}\int_{\mathcal{U}^h} \mbox{dist}^2\left(\nabla
  (u^h\circ \phi^{-1}) (\nabla h A^{-1}), SO(3)\right) |\det \nabla
h|^{-1}~\mbox{d}z \\
& \geq c\frac{1}{h}\int_{\mathcal{U}^h} \mbox{dist}^2\left(\nabla
  (u^h\circ \phi^{-1}), SO(3)\right) ~\mbox{d}z + \mathcal{O}(h^4),
\end{split}
\end{equation}
where we used that:
$$\mbox{dist}^2\left(F(\mbox{Id}_3+\mathcal{O}(x_3^2)), SO(3)\right)
\geq c~\mbox{dist}^2\left(F,
  SO(3)(\mbox{Id}_3+\mathcal{O}(x_3^2))^{-1}\right) \geq
c~\mbox{dist}^2\left(F, SO(3)\right) + \mathcal{O}(x_3^4).$$
By the geometric rigidity estimate \cite{FJMhier}, and since $\det\nabla \phi$ is uniformly bounded away
from $0$ in $\Omega$, the estimate (\ref{gesa}) becomes:
\begin{equation}\label{gesb}
\begin{split}
E^h(u^h) & \geq c_h\frac{1}{h}\int_{\mathcal{U}^h} \left |\nabla
  (u^h\circ \phi^{-1}) - R_h\right|^2 ~\mbox{d}z + \mathcal{O}(h^4) \geq 
c_h\frac{1}{h}\int_{\Omega^h} \left |\nabla u^h - R_h\nabla \phi\right|^2
+ \mathcal{O}(h^4) \\ & \geq c_h\frac{1}{h}\int_{\Omega^h} \left
  |\nabla u^h - R_h(Q+x_3B)\right|^2 ~\mbox{d}x + \mathcal{O}(h^4), 
\end{split}
\end{equation}
for some $R_h\in SO(3)$ and the constants $c_h$ depending on the
domains $\mathcal{U}^h$. However, if we replace the integration on
$\Omega^h$ in (\ref{gesa}), (\ref{gesb}), by integration on a small
cube $(-\frac{h}{2}, \frac{h}{2})^3$, all the subsequent constants,
including $c_h$ from the geometric rigidity estimate, will be uniform
and independent of $\phi$. This leads, by the well known technique of
approximation \cite{FJMhier, lemapa1, lepa}, to the following result:

\begin{lemma}\label{krowa}
Assume (\ref{ges2}). There exists a sequence of rotation fields
$R^h\in W^{1,2}(\Omega, SO(3))$, with:
\begin{equation}\label{gesp}
\frac{1}{h} \int_{\Omega^h} |\nabla u^h - R^h(Q+x_3 B)|^2~\mathrm{d}x
\leq Ch^4 \quad \mbox{ and } \quad \int_\Omega |\nabla R^h|^2\leq Ch^2.
\end{equation}
\end{lemma}
The proof of Lemma \ref{krowa}  reproduces the lines of Theorem 1.6 in
\cite{lemapa1}, hence we omit it. Note that the above is parallel to Lemma \ref{approx}, now with an
improved accuracy of the error bound due to the smaller scaling of
the energy in (\ref{ges2}). With the help of (\ref{gesp}) we now
derive the limiting properties of the sequence $u^h$.

\medskip

{\bf 2.} Define $\bar R^h=\mathbb{P}_{SO(3)} \fint_{\Omega^h} (\nabla
u^h)Q^{-1}~\mbox{d}x$. Since: 
$\mbox{dist}^2(\fint (\nabla u^h)Q^{-1}, SO(3)) \leq |\fint (\nabla
u^h)Q^{-1} - R^h(x')|^2$ for every $x'\in\Omega$, it follows upon integrating and using the Poincar\`e
inequality, together with (\ref{gesp}), that the projection in $\bar
R^h$ is well defined. Indeed:
\begin{equation*}
\begin{split}
\mbox{dist}^2(\fint_{\Omega^h} (\nabla u^h)Q^{-1}, SO(3)) & \leq 
\fint |\fint (\nabla u^h)Q^{-1}~\mbox{d}x - R^h(x')|^2~\mbox{d}x' \\ &
\leq C\left(  \fint |\fint (\nabla u^h)Q^{-1} - \fint R^h|^2 +   \fint |R^h - \fint R^h|^2\right)
\\ & \leq C\left(  \fint_{\Omega^h} | (\nabla u^h)Q^{-1} - R^h|^2 +   \fint_{\Omega^h}
  |\nabla R^h|^2\right)\leq Ch^2.
\end{split}
\end{equation*}
In a similar manner as above and again by (\ref{gesp}), we observe that:
\begin{equation}\label{gesc}
\begin{split}
& \int_\Omega |R^h - \bar R^h|^2 \leq \int_\Omega |R^h - \fint (\nabla
u^h)Q^{-1} |^2 \leq Ch^2, \\ &
\|(\bar R^h)^T R^h - \mbox{Id}_3\|^2_{W^{1,2}(\Omega)} \leq Ch^2.
\end{split}
\end{equation}
In particular, we have the following convergence to some matrix field
$S$ on $\Omega^1$, up to a subsequence: 
\begin{equation}\label{gesk}
\frac{1}{h} \big((\bar R^h)^T R^h - \mbox{Id}_3\big) \rightharpoonup S
\qquad \mbox{ weakly in } W^{1,2}(\Omega^1, \mathbb{R}^{3\times 3}),
\end{equation}

\medskip

{\bf 3.} Define the renormalised deformations $y^h\in W^{1,2}(\Omega^1,
\mathbb{R}^3)$ by $y^h(x', x_3) = (\bar R^h)^T u^h(x', hx_3) - c^h$,
where $c^h$ is a constant vector such that by the Poincar\`e inequality
and in view of:
$$\int_{\Omega^1} |(\bar R^h)^T\nabla u^h(x', hx_3) - Q|^2 \leq
C\frac{1}{h} \int_{\Omega^h} |\nabla u^h - \bar R^h Q|^2 \leq Ch^2,$$
there holds:
$$ y^h \to y \quad \mbox{ and } \quad \frac{1}{h} \partial_3y^h \to
\vec b \qquad \mbox{ in } W^{1,2}(\Omega^1, \mathbb{R}^3).$$ 
The above is, naturally, consistent with  the results of Theorem \ref{thm2} (i).
We now perform the analysis similar to step 5 in the proof of Theorem
\ref{thm2}. Define the rescaled strains:
$$\mathcal{G}^h(x', x_3) = \frac{1}{h^2} \Big((R^h)^T \nabla u^h(x',
hx_3) - (Q+hx_3B)\Big)A^{-1}.$$
Note that in view of (\ref{gesp}), $\{\mathcal{G}^h\}$ is uniformly
bounded in $L^2$, so that up to a subsequence:
\begin{equation}\label{mewa}
\mathcal{G}^h \rightharpoonup \mathcal G \quad \mbox{ weakly in } L^2(\Omega^1, \mathbb{R}^{3\times 3}).
\end{equation}

Consider the vector fields:
\begin{equation*}
\begin{split}
\frac{1}{h^2}&\left( \partial_3 y^h - h(\vec b + hx_3\vec d)\right)  = 
\frac{1}{h}\left( (\bar R^h)^T\nabla u^h(x', hx_3) - (Q +
  hx_3B)\right)e_3 \\ & \qquad = 
\frac{1}{h}(\bar R^h)^T\left(\nabla u^h(x', hx_3) - R^h(Q +
  hx_3B)\right)e_3  + \frac{1}{h}\left( (\bar R^h)^TR^h - \mbox{Id}\right)(Q + hx_3B)e_3. 
\end{split}
\end{equation*}
The first term in the right hand side above converges in $L^{2}(\Omega^1)$ to $0$ by
(\ref{gesp}), while the second term converges to $SQe_3 = S\vec b$ by
(\ref{gesk}).
Hence we observe the same convergence of the difference quotients:
\begin{equation*}
\begin{split}
f^{s,h}(x', x_3) & = \frac{1}{h^2}\frac{1}{s} \left(y^h(x', x_3+s) -
  y^h(x', x_3) - hs(\vec b + \big(hx_3 + \frac{hs}{2})\vec d\big)\right) \\ & = 
\frac{1}{h^2} \fint_0^s \partial_3 y^h(x', x_3+t) - h\big(\vec b + h(x_3 + t)
\vec d\big)~\mbox{d}t,
\end{split}
\end{equation*}
and of their $\partial_3$ derivatives, namely:
\begin{equation*}
f^{s,h} \to S\vec b \quad \mbox{ and } \quad 
\partial_3 f^{s,h} \to 0\qquad \mbox{ in } L^{2}(\Omega^1, \mathbb{R}^3).
\end{equation*}
Regarding the in-plane derivatives, for $i=1,2$ we have:
\begin{equation*}
\begin{split}
\partial_i &f^{s,h}(x', x_3) \\ & = \frac{1}{h^2s}\Big((\bar R^h)^T\partial_i u^h(x', hx_3+hs) -
(\bar R^h)^T\partial_i u^h(x', hx_3) - hs\big(\partial_i\vec b +
h(x_3+\frac{s}{2})\partial_i\vec d\big)\Big)\\ & 
= \frac{1}{s}\Big((\bar R^h)^T R^h \mathcal{G}^h(x', x_3+s) - (\bar
R^h)^T R^h \mathcal{G}^h(x', x_3) \Big)e_i \\ & \qquad\qquad
\qquad\qquad  \qquad\qquad + \frac{1}{h} \big((\bar
R^h)^T R^h - \mbox{Id}_3\big)B e_i - (x_3+\frac{s}{2})\partial_i \vec d.
\end{split}
\end{equation*}
Therefore, by (\ref{mewa}) and (\ref{gesk}):
\begin{equation*}
\partial_i f^{s,h} \rightharpoonup  \frac{1}{s}\Big(\mathcal{G}(x', x_3+s)
- \mathcal{G}(x', x_3) \Big)e_i + SBe_i - (x_3+\frac{s}{2})\partial_i
\vec d \qquad  \mbox{ weakly in } L^{2}(\Omega^1, \mathbb{R}^3).
\end{equation*}
Concluding, $f^{s,h}$ converges weakly (up to a subsequence) in
$W^{1,2}(\Omega^1)$ to $S\vec b$, and the limit in the right hand side
above must coincide with $\partial_i (S\vec b)$, yielding:
\begin{equation}\label{gest}
\mathcal{G}(x', x_3)e_i - \mathcal{G}(x',0)e_i = x_3\big( \partial_i(S\vec
b) - SBe_i\big) + \frac{x_3^2}{2} \partial_i\vec d \qquad \mbox{ for } i=1,2.
\end{equation}

\medskip

{\bf 4.} We now estimate the rescaled energies, as desired.
Firstly, observe that for $(x', x_3)\in\Omega^1$:
$$W\big((\nabla u^h(x', hx_3)) A^{-1}\big) = W\big((R^h(x'))^T (\nabla u^h(x', hx_3)) A^{-1}\big) =
W\big(\mbox{Id}_3 + h^2\mathcal{G}^h(x', x_3) + hx_3 B(x')A^{-1}\big).$$
Since $BA^{-1}\in so(3)$, it follows by frame invariance that:
\begin{equation*}
\begin{split}
W\big((\nabla &u^h(x', hx_3)) A^{-1}\big)  = W\Big( (e^{-hx_3
  BA^{-1}})(\mbox{Id}_3 + h^2\mathcal{G}^h(x', x_3) + hx_3
B(x')A^{-1})\Big) \\ & = W\Big( \big(\mbox{Id}-
hx_3BA^{-1}+\frac{h^2x_3^2}{2} (BA^{-1})^2 +
\mathcal{O}(h^3)\big)\big(\mbox{Id}_3 + h^2\mathcal{G}^h(x', x_3) + hx_3 
B(x')A^{-1}\big)\Big) \\ & = W\Big( \mbox{Id}_3 + h^2\mathcal{G}^h 
- \frac{h^2x_3^2}{2} (BA^{-1})^2 + \mathcal{O}(h)h^2\mathcal{G}^h +\mathcal{O}(h^3)\Big).
\end{split}
\end{equation*}
Define the sets $\Omega_h = \{x\in\Omega^1;~ h|\mathcal{G}^h(x)|\leq 1\}$.
Using Taylor's expansion, we now obtain:
\begin{equation*}
\begin{split}
\liminf_{h\to 0}\frac{1}{h^4}E^h(u^h) & \geq \liminf_{h\to
  0}\frac{1}{h^4} \int_{\Omega^1} \chi_{\Omega_h} W\Big( \mbox{Id}_3 + h^2\mathcal{G}^h 
- \frac{h^2x_3^2}{2} (BA^{-1})^2 + \mathcal{O}(h^3)\Big)~\mbox{d}x 
\\ & = \liminf_{h\to 0}\frac{1}{2}\int_{\Omega^1} \mathcal{Q}_3\Big(
\chi_{\Omega_h}(\mathcal{G}^h - \frac{x_3^2}{2}
(BA^{-1})^2 \Big) \geq \frac{1}{2}\int_{\Omega^1} \mathcal{Q}_3\Big(
\mathcal{G} - \frac{x_3^2}{2} (BA^{-1})^2 \Big)  \\ & 
\geq \frac{1}{2}\int_{\Omega^1} \mathcal{Q}_2\Big(
\mathcal{G}(x', x_3)_{2\times 2} + \frac{x_3^2}{2}
\frac{1}{4\lambda}\nabla\lambda\otimes \nabla\lambda\Big) ~\mbox{d}x,
\end{split}
\end{equation*}
where we used the weak sequential lower-semincontinuity of the quadratic form
$\mathcal{Q}_3$ in $L^2$, and since
by (\ref{mewa}) $\chi_{\Omega_h} \to 1$ in $L^1(\Omega^1)$, 
resulting in $\chi_{\Omega_h} \mbox{sym } \mathcal{G}^h \rightharpoonup
\mbox{sym } \mathcal{G}$ weakly in $L^2(\Omega^1)$. 

By (\ref{gest}) we hence arrive at:
\begin{equation*}
\begin{split}
\liminf_{h\to 0}&\frac{1}{h^4}E^h(u^h)  
\geq \frac{1}{2}\int_{\Omega^1} \mathcal{Q}_2\Big(
\mathcal{G}(x',0)_{2\times 2} + x_3 \big(\nabla(S\vec b)
- SB\big)_{2\times 2} - \frac{x_3^2}{4} M\Big) ~\mbox{d}x \\ &
\geq \frac{1}{2}\int_{\Omega^1} \mathcal{Q}_2\Big(
\mathcal{G}(x',0)_{2\times 2} - \frac{x_3^2}{4} M\Big) ~\mbox{d}x \\ &
= \frac{1}{2}\Big(\int_{\Omega^1} \mathcal{Q}_2(
\mathcal{G}(x',0)_{2\times 2}) - \int_{\Omega^1}
\frac{x_3^2}{2}\left\langle {L}_2\big(\mathcal{G}(x',0)_{2\times 2}\big) :
M\right\rangle + \int_{\Omega^1} \frac{x_3^4}{16}\mathcal{Q}_2(M)\Big)
\\ & = \frac{1}{2}\Big(\int_{\Omega} \mathcal{Q}_2(
\mathcal{G}(x',0)_{2\times 2}) - \frac{1}{24}\int_{\Omega}
\left\langle {L}_2\big(\mathcal{G}(x',0)_{2\times 2}) : M\right\rangle
+ \frac{1}{16}(\frac{1}{144} + \frac{1}{180}) \int_{\Omega}
\mathcal{Q}_2(M)~\mbox{d}x' \Big) \\ &
= \frac{1}{2}\Big(\int_{\Omega} \mathcal{Q}_2\big(
\mathcal{G}(x',0)_{2\times 2} - \frac{1}{48} M\big) ~\mbox{d}x'
+ \frac{1}{16}\frac{1}{180} \int_{\Omega} \mathcal{Q}_2(M)~\mbox{d}x'
\Big)\\ & \geq c \int_{\Omega} \mathcal{Q}_2(M)~\mbox{d}x'.
\end{split}
\end{equation*}
Above, $L_2:\mathbb{R}^{2\times 2}\to\mathbb{R}^{2\times 2}$ stands
for  the linear map with the property: $\mathcal{Q}_2(F) = \langle L_2(F)
: F\rangle$ and  $\langle L_2(F) : \tilde F\rangle = \langle
L_2(\tilde F) : F\rangle$ for all $F, \tilde F\in \mathbb{R}^{2\times 2}.$
The proof of Theorem \ref{optimal} is now complete in view of the nondegeneracy of
the quadratic form $\mathcal{Q}_2$ on  symmetric matrices.
\end{proof}
\vspace{\baselineskip}

\begin{center}
{\bf Part C: Application}
\end{center}

\section{Application to liquid crystal glass }\label{lg}

Nematic liquid crystals elastomers are rubber-like, cross-linked,
polymeric solids, which have both positional elasticity, due to
rubber-like, solid response of the polymer chains, 
and the  orientation elasticity due to the separately deforming director.
A nematic liquid crystal glass is a very highly
cross-linked nematic liquid crystal elastomer such 
that the director is effectively constrained to move with the liquid
crystal elastomer matrix. 

In this section we consider a model of nematic glass \cite{Warner, MWB} whose
referential conformation $A$ corresponds to a prolate
ellipsoid, elongating the eigenvector $\vec n$ by the factor
$\lambda$, while shrinking the invariant $2$d subspace $\vec n^\perp =
\mbox{span}(v, w)$ by factor $\lambda^\nu$:
$$A= \lambda^{-\nu} v\otimes v + \lambda^{-\nu} w\otimes w + \lambda \vec n\otimes \vec n =
\lambda^{-\nu}\big(\mbox{Id}_3 + (\lambda^{\nu+1}-1)\vec n\otimes \vec
n\big), \qquad \lambda>1, \quad |\vec n|=1.$$
In other circumstances,  $A$ corresponds to a contraction $\lambda$ in
direction of $\vec n$ and an expansion $\lambda^{-\nu}$ in the
perpendicular directions, and so $\lambda$ could also be 
less than 1 \cite{modes2011blueprinting}. The coefficient $\nu$
is experimentally verified to be in the range $\frac{1}{2}<\nu<2$.
Setting $r=\lambda^{\nu +1}$, and writing $\lambda^{-\nu} =
r^\delta$ with $\delta= -\frac{\nu}{\nu+1}$, we obtain
the metric $G$ and its  symmetric square root $A=\sqrt{G}$ given by:
\begin{equation}\label{lc}
G(x', x_3) =G(x') = r^{2\delta} (\mbox{Id}_3 + (r^2-1)\vec n\otimes
\vec n), \qquad A(x') = r^\delta (\mbox{Id}_3 + (r-1)\vec n\otimes \vec n).
\end{equation}

\medskip

We start by the following observation:

\begin{theorem}\label{lem71}
Assume that: 
\begin{equation}\label{flat}
\vec n \in S^1 \quad \mbox{i.e. } \vec n = (n_1, n_2, 0)^T\in S^2,
\mbox{ with } n=(n_1, n_2)\in S^1.
\end{equation}
Then the following conditions are equivalent:
\begin{itemize}
\item[(i)] the metric $G$ as in (\ref{lc}) is immersable, i.e. $G=(\nabla u)^T\nabla u$ for
some smooth $u:\Omega^1\rightarrow \mathbb{R}^3$,
\item[(ii)] the Gaussian curvature $\kappa (\mathrm{Id}_2 + (r^2-1) n\otimes n)$ vanishes
identically in $\Omega$,
\item[(iii)] $\mathrm{curl}^T\mathrm{curl }~ G_{2\times 2} = 0$,
\item[(iv)] the following curvatures of $G$ vanish: $R^3_{112} =
  R^3_{221} = R_{1212} = 0$.
\end{itemize}
\end{theorem}
\begin{proof}
It is clear that (i) holds iff the Riemann tensor of $G$ vanishes, which is equivalent to:
$$ \kappa = \kappa (\mbox{Id}_2 + (r^2-1) n \otimes n)=0.$$
We now calculate the Gaussian curvature $\kappa$ of the $2$d metric
$ \mathrm{Id}_2 + (r^2-1) n \otimes n$ and prove that:
$$ \kappa = r^{2\delta} \kappa(G_{2\times2})  = \frac{1}{2} (r^{-2} -1) \mbox{curl}^T\mbox{curl }
(n\otimes n).$$
This will achieve the lemma, because $R^3_{112}=R^3_{221}=0$
automatically, while $R_{1212}=0$ is equivalent to (ii).
We write $r^2-1=\gamma > 0$ and compute:
\begin{equation*}
\begin{split}
\big(\det (\mathrm{Id}_2+ \gamma n \otimes n)\big)^2 \kappa = &
\det\left[\begin{array}{ccc} q & \frac{1}{2} e_{,1} & f_{,1} - \frac{1}{2}e_{,2}  \\
f_{,2} - \frac{1}{2}g_{,1} & e & f \\
\frac{1}{2}g_{,2} & f & g \end{array}\right] \\
& - \det \left[\begin{array}{ccc} 0 & \frac{1}{2} e_{,2} & g_{,1}  \\
\frac{1}{2} e_{,2} - \frac{1}{2}g_{,1} & e & f \\
\frac{1}{2}g_{,1} & f & g \end{array}\right],
\end{split}
\end{equation*}
where $q= -\frac{1}{2} e_{,22} + f_{,12} - \frac{1}{2}g_{,11}$ and
$e=1+\gamma n_1^2$, $f=\gamma n_1n_2$, $g=1+\gamma n_2^2$.
A direct calculation now gives that the right hand side above equals:
\begin{equation*}
\begin{split}
(1+\gamma)q + \gamma^3 \cdot 0 + \gamma^2 (& -n_2^2 n_{1,1} n_{2,2} -
n_1n_2n_{2,1} n_{2,2} + n_1n_2 n_{1,2} n_{2,2} - n_1n_2 n_{1,1}
n_{1,2} \\
& - n_1^2 n_{1,1} n_{2,2}  + n_1n_2 n_{1,1} n_{2,1} +
n_2^2n_{2,1}^2 + n_1^2 n_{1,2}^2) = (1+\gamma)q.
\end{split}
\end{equation*}
The equality above follows since all the terms in the bracket
multiplying $\gamma^2$ cancel out. 
This can be easily seen by substituting $(n_1, n_2) = (\cos\theta, \sin \theta)$
for some angle function $\theta:\Omega\rightarrow \mathbb{R}$.

We hence see that $\kappa = 0$ iff $q=0$. But note that:
$$q= -\frac{\gamma}{2} \big( (n_1^2)_{,22} - 2(n_1n_2)_{,12} +
(n_2^2)_{,11}\big) = -\frac{\gamma}{2} \mbox{curl}^T \mbox{curl} (n\otimes n).$$
Since $\det (\mathrm{Id}_2+ (r^2-1) n \otimes n) = 1+\gamma$, it
follows that $\kappa = -\frac{1}{2} \frac{\gamma}{\gamma+1}
\mbox{curl}^T \mbox{curl} (n\otimes n).$ 
\end{proof}

\begin{example}\label{kaushik}
In accordance with (\ref{lc}), the following metric has been put forward in \cite{MWB2} for the description
of disclination-mediated thermo-optical response in nematic glass sheets:
$G(x', x_3) = \alpha \mbox{Id}_3 + \beta \vec n(x') \otimes \vec
n(x')$, where $\alpha, \beta > 0$ are constants, and $\vec n$ is as in
(\ref{flat}) with:
$$n_1 = \cos (\theta + \psi),\quad n_2 = \sin (\theta + \psi), \qquad 
\theta = \arctan \frac{x_2}{x_1}, \quad \psi\equiv const.$$
Note that $\theta$ is the polar angle and so setting the constant $\psi = 0$
gives the radial pattern, while $\psi = \pi/2$ gives the azimuthal pattern, and
other values of $\psi$ yield spiral patterns.
It is easy to check that $\mbox{curl}^T\mbox{curl} (n\otimes n) =
0$. Therefore,  if the simply connected $\Omega$ does not contain $0$
(since $0$ is a singularity for $G$), then the metrics $G_{2 \times 2}$ and $G$ are
immersable by Lemma \ref{lem71}, and thus $\inf E^h = 0$.
However, one may not have a global immersion (implying hence a
higher energy scaling -- see \cite{MWB} for a construction with $h^2$ scaling) if $0\in \Omega$. 
\end{example}

\begin{remark}
Consider any 2d metric $G_{2\times 2}$ with constant eigenvalues $0<
\lambda_1\leq \lambda_2$: 
$$G_{2\times 2} = \lambda_1 v\otimes v + \lambda_2 w\otimes w = \lambda_1
(\mbox{Id}_2- \frac{\lambda_2-\lambda_1}{\lambda_1} w\otimes w). $$
We see that such $G_{2\times 2}$ is flat iff 
$\mbox{curl}^T \mbox{curl} (G_{2\times 2})=0$. Interestingly,
$\mbox{curl}^T \mbox{curl} $ is the leading order term in the
expansion of the Gaussian curvature of a $2$d metric at $\mbox{Id}_2$.
\end{remark}

In the $2$d case as in (\ref{flat}), we directly obtain:

\begin{theorem}
Assume that $G$ is as in (\ref{lc}) with (\ref{flat}).
Then, Theorems \ref{thm2} and \ref{thm3} hold with the Cosserat vector
$\vec b$ given by:
$$\vec b = r^\delta \vec N$$ 
and with the limiting functional:
\begin{equation*}\label{lqfun}
\begin{split}
\mathcal{I}_G(y) = \mathcal{I}_{\vec n}(y) & = 
 \frac{1}{24} \int_\Omega
\mathcal{Q}_2\left(x', r^{\delta} (\nabla y)^T\nabla \vec
  N\right)~\mathrm{d}x' \\ & = 
\frac{1}{24}\int_\Omega \mathcal{Q}_2^0\left(r^{\delta} (A_{2\times 2})^{-1} (\nabla y)^T \nabla \vec
  N (A_{2\times 2})^{-1} \right)~\mathrm{d}x'.
\end{split}
\end{equation*}
Denote:
$$\alpha=\frac{r-1}{r}.$$
Then: 
$$ (A_{2\times 2})^{-1}  = \frac{1}{r^{1+\delta}} \left(\mathrm{Id}_2 +
(r-1) n^\perp\otimes n^\perp\right) = \frac{1}{r^{\delta}}\left(\mathrm{Id}_2 -
\alpha ~ n\otimes n\right)$$
and the quadratic form in the second integrand in (\ref{lqfun}) can be equivalently expressed as:
\begin{equation*}
\begin{split}
r^{\delta} (A_{2\times 2})^{-1} &(\nabla y)^T \nabla \vec
  N (A_{2\times 2})^{-1} \\ & = r^{-\delta}\left( (\nabla y)^T\nabla \vec
    N - 2 \alpha \mathrm{sym}\big((n\otimes \partial_n y)\nabla \vec N\big)
  + \alpha^2 \left\langle \partial_n y, \partial_n\vec N\right\rangle n\otimes n\right).
\end{split}
\end{equation*}
Moreover, when $W$ is isotropic  so that (\ref{Qiso}) holds, we have:
\begin{equation}\label{nowe}
\forall F_{2\times 2}\in\mathbb{R}^{2\times 2}_{sym} \qquad 
\mathcal{Q}_2(x', F_{2\times 2}) =
\frac{1}{r^{4\delta}}\mathcal{Q}^0_{2, iso} \Big((\mathrm{Id}_2 -
\alpha ~ n\otimes n)F_{2\times 2}(\mathrm{Id}_2 -\alpha ~ n\otimes n)\Big)
\end{equation}
\end{theorem}

\bigskip

We now turn to the general case of the general $3$d director $\vec n$.

\begin{theorem}\label{th3d}
Assume that $G$ is of the form (\ref{lc}). Let $n=(n_1,
n_2)\in \mathbb{R}^2$ denote the tangential component of the director
vector $\vec n$.  Denote also:
$$\gamma = \frac{1}{n_3^2+|n|^2r^2}.$$
Then Theorems \ref{thm2} and \ref{thm3} hold with the Cosserat vector $\vec b$ is given by:
\begin{equation}\label{lcb}
\vec b = (r^2-1)n_3\gamma (\partial_n y) +
r\sqrt{\gamma} (r^\delta\vec N)
= \frac{(r^2-1)n_3}{n_3^2 + |n|^2r^2} \partial_n y +
\frac{r^{1+\delta}}{\sqrt{n_3^2 + |n|^2r^2}}\vec N
\end{equation}
and with the limiting functional:
\begin{equation*}
\begin{split}
\mathcal{I}_G(y) = \mathcal{I}_{\vec n}(y) & = 
 \frac{1}{24} \int_\Omega
\mathcal{Q}_2\left(x', (\nabla y)^T\nabla \vec
  b\right)~\mathrm{d}x'.
\end{split}
\end{equation*}
Moreover, when $W$ is isotropic so that (\ref{Qiso}) holds, we have:
\begin{equation}\label{lcQ1}
\begin{split}
\forall F_{2\times 2}\in\mathbb{R}^{2\times 2}_{sym} \qquad 
& \mathcal{Q}_2(x', F_{2\times 2}) = \\
 & = \mu\frac{1}{r^{4\delta}}\Big(|F_{2\times 2}|^2 - 2\big((r^2-1)\gamma\big) |F_{2\times
    2} n|^2 + \big((r^2-1)\gamma\big)^2\langle F_{2\times 2} n, n\rangle^2\Big) \\ &
+ \frac{\lambda\mu}{\lambda+\mu} \frac{1}{r^{4\delta}}\Big(
\mathrm
{tr}F_{2\times 2} - \big((r^2-1)\gamma\big) \langle F_{2\times 2}n,  n\rangle\Big)^2.
\end{split}
\end{equation}
Setting $\tilde \gamma = \frac{1}{|n|^2} \left(1 - \sqrt{\gamma} \right)$,
the above formula is equivalent to:
\begin{equation}\label{lcQ2}
\mathcal{Q}_2(x', F_{2\times 2}) = 
\frac{1}{r^{4\delta}}\begin{cases}\mathcal{Q}^0_{2, iso}\Big( 
(\mathrm{Id}_2 -
\tilde \gamma ~ n\otimes n)F_{2\times 2}(\mathrm{Id}_2 -\tilde\gamma ~
n\otimes n)\Big)&\text{ if }n_3(x')^2<1,\\ 
\mathcal{Q}^0_{2, iso}(F_{2\times2})&\text{ if }n_3(x')^2=1.
\end{cases}
\end{equation}
\end{theorem}
\begin{proof}
We easily compute:
\begin{equation*}
\begin{split}
&\det G = r^{2+6\delta}, \qquad \det G_{2\times 2} = r^{4\delta}(r^2 -
(r^2-1)n_3^2) \\
&(G_{2\times 2})^{-1} = \frac{1}{r^{2\delta}}(n_3^2 +
|n|^2 r^2)^{-1}\left(\mbox{Id} - (r^2-1) n^\perp\otimes n^\perp\right),
\end{split}
\end{equation*}
where $n^\perp = (n_1, n_2)^\perp = (-n_2, n_1)$. Therefore:
$$(G_{2\times 2})^{-1}\left[\begin{array}{c}G_{13}\\G_{23}\end{array}\right] = 
\frac{(r^2-1)n_3}{n_3^2 + |n|^2 r^2} n$$
which implies the formula (\ref{lcb}).

To prove (\ref{lcQ2}) in view of Theorem \ref{isobest}, it is now enough to check
directly that the positive definite matrix $r^{-\delta}
\left(\mbox{Id}_2 - \tilde\gamma n\otimes n\right)$ equals
$\sqrt{G_{2\times 2}}^{-1}$.
Indeed:
\begin{equation*}
\begin{split}
\left(\mbox{Id}_2 - \tilde\gamma n\otimes n\right)^2
\left(\mbox{Id}_2 + (r^2-1) n\otimes n\right) & = 
\left(\mbox{Id}_2 +(\tilde\gamma |n|^2 - 2\tilde\gamma) n\otimes n\right)
\left(\mbox{Id}_2 + (r^2-1) n\otimes n\right) \\ & = 
\left(\mbox{Id}_2 - (r^2-1)\gamma n\otimes n\right)
\left(\mbox{Id}_2 + (r^2-1) n\otimes n\right) =\mbox{Id},
\end{split}
\end{equation*}
as $\tilde\gamma^2|n|^2 - 2\tilde\gamma = -(r^2-1)\gamma. $
\end{proof}

\begin{remark}
The expression in (\ref{nowe}) is consistent with  (\ref{lcQ2}), as
for $n_3=0$ it follows that $\gamma = \frac{r^2-1}{r^2}$ and
$\tilde\gamma = 1-1/r = \alpha$.
The expression in (\ref{nowe}) is also consistent with  Remark
  \ref{rem4.3}, in the following sence. Take $\vec n=e_3$. Then
  $D= r^{-2\delta} \mbox{diag}(1,1,r^{-1}) F^*_{2\times 2}
  \mbox{diag}(1,1,r^{-1}) =r^{-2\delta} F^*_{2\times 2}$.
Hence, by (\ref{rem}):
$$\mathcal{Q}_2(x', F_{2\times 2}) = \frac{1}{r^{4\delta}} \left(
\mu |F_{2\times 2}|^2 + \frac{\lambda\mu}{\lambda+\mu}
|\mbox{tr}F_{2\times 2}|^2 \right), $$
while  (\ref{lcQ1}), (\ref{lcQ2}) give the same formula directly.
\end{remark}

\end{document}